\pdfoutput=1
\documentclass{amsart}
\usepackage{amsmath}
\usepackage{amssymb}
\usepackage{graphicx}
\usepackage{tabularx}

\theoremstyle{plain}
\newtheorem{theorem}{Theorem}
\newtheorem{proposition}{Proposition}
\newtheorem{corollary}{Corollary}
\newtheorem{fact}{Fact}

\theoremstyle{definition}

\newtheorem{remark}{Remark}

\newcommand{\HH}{\begin{picture}(10,5)
\put(5,3){\circle{7}}
\qbezier(1.8,4)(5,4)(8,4)
\qbezier(1.8,2)(5,2)(8,2)
\qbezier(4,-0.2)(4,3)(4,6)
\qbezier(6,-0.2)(6,3)(6,6)
\end{picture}
}

\newcommand{\h}{\begin{picture}(10,5) 
\put(5,3){\circle{7}} 
\qbezier(1.5,3)(5,3)(8.5,3) 
\qbezier(4,0)(4,3)(4,6)
\qbezier(6,0)(6,3)(6,6)
\end{picture}}
\newcommand{\tr}{\begin{picture}(10,5)
\put(5,3){\circle{7}}
\qbezier(5,0)(5,3)(5,6)
\qbezier(2,1)(5,3)(8,5)
\qbezier(8,1)(5,3)(2,5)
\end{picture}}
\newcommand{\bigcircle}{\begin{picture}(10,5) \put(5,3){\circle{8}} \end{picture}}
\newcommand{\thr}{\begin{picture}(10,5)
\put(5,3){\circle{7}}
\qbezier(1.5,3)(5,3)(8.5,3)
\qbezier(3,0)(3,3)(3,5.5)
\qbezier(5,-0.5)(5,3)(5,6.5)
\qbezier(7,0)(7,3)(7,5.5)
\end{picture}}

\newcommand{\double}{\begin{picture}(10,5) \put(0,0){\vector(1,1){10}} \put(10,0){\vector(-1,1){10}}
 \end{picture}}

\begin{document}
\title[Sub-chord diagrams of knot projections]{Sub-chord diagrams of knot projections}
\author[N. Ito]{Noboru Ito}
\address{Waseda Institute for Advanced Study, 1-6-1 Nishi-Waseda Shinjuku-ku Tokyo 169-8050 Japan, TEL: +81-3-5286-8392}
\email{noboru@moegi.waseda.jp}
\author[Y. Takimura]{Yusuke Takimura}
\address{Gakushuin Boys' Junior High School, 1-5-1 Mejiro Toshima-ku Tokyo 171-0031 Japan}
\email{Yusuke.Takimura@gakushuin.ac.jp}
\begin{abstract}
A chord diagram is a circle with paired points with each pair of points connected by a chord.  Every generic immersed spherical curve provides a chord diagram by associating each chord with two preimages of a double point.  Any two spherical curves can be related by a finite sequence of three types of local replacement RI, RI\!I, and RI\!I\!I, called Reidemeister moves.  This study counts the difference in the numbers of sub-chord diagrams embedded in a full chord diagram of any spherical curve by applying one of the moves RI, strong RI\!I, weak RI\!I, strong RI\!I\!I, and weak RI\!I\!I defined by connections of branches related to the local replacements (Theorem~\ref{thm_sec4}).  This yields a new integer-valued invariant under RI and strong RI\!I\!I that provides a complete classification of prime reduced spherical curves with up to at least seven double points (Theorem~\ref{main2_thm}, Fig.~\ref{hyou5}): there has been no such invariant before.  The invariant expresses the necessary and sufficient condition that spherical curves can be related to a simple closed curve by a finite sequence of RI and strong RI\!I\!I moves (Theorem~\ref{main3}).  Moreover, invariants of spherical curves under flypes are provided by counting sub-chord diagrams (Theorem~\ref{flype_thm}).  
%Moreover, we obtain new invariants of flypes (Theorem \ref{flype_thm}).  We also mention new properties of one of Arnold invariants, called the average invariant, of spherical curves (Sec.~\ref{average_sec}).  In the last, we obtain a tables of prime reduced spherical curves providing the number of some sub-chord diagrams, the values of the new invariant, and the average invariant.  
%The number of a chord diagram often provides the topological properties of spherical curves.  
\end{abstract}
\maketitle

\section{Introduction.}
Any two knot projections (equivalently, generic immersed spherical curves) are related by a finite sequence of three types of local replacement, RI, RI\!I, and RI\!I\!I, called {\it{Reidemeister moves}}, on knot projections.  These replacements are defined by Fig.~\ref{ych5}.  
\begin{figure}[h!]
\includegraphics[width=8cm]{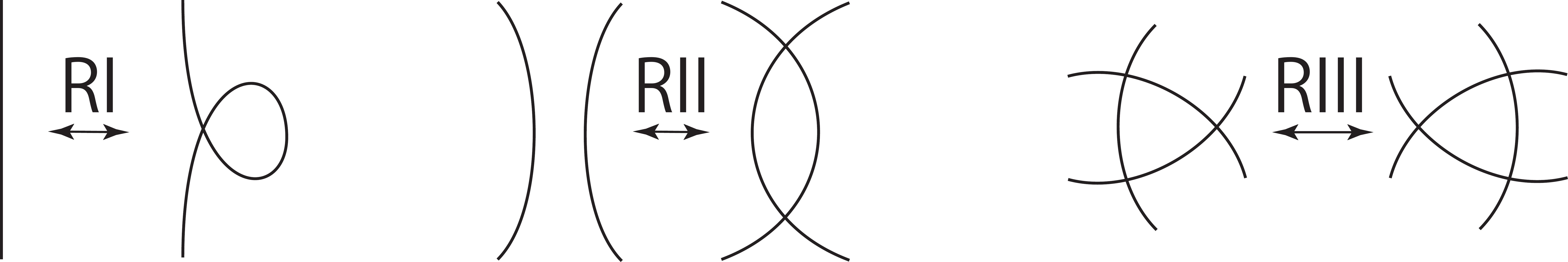}
\caption{Reidemeister moves RI, RI\!I, and RI\!I\!I from the left.}\label{ych5}
\end{figure}

A {\it{chord diagram}} is a circle with chords with endpoints at different places on the circle.  
%two endpoints of each chord on the circle.  
Chord diagrams are often used to study knots or knot projections.  {\it{A chord diagram}} $CD_P$ {\it{of a knot projection}}, $P$, is a circle with the preimages of double points for which every pair of double-point preimages is connected by an arc.  Examples are shown in the leftmost columns of Figs.~\ref{hyou3} and \ref{hyou4}.  In this paper, a {\it{sub-chord diagram}} of $CD_P$ is a chord diagrams embedded in $CD_P$.   

This paper exhibits characteristics of chord diagrams $\otimes$, $\tr$, $\h$, and $\thr$ under five types of Reidemeister moves.  In particular, we split the second (resp.,~third) Reidemeister move into the strong and weak second (resp.,~third) Reidemeister moves, as follows.  
%The chord diagrams $\otimes$, $\tr$, $\h$, $\thr$ are called the cross chord, triple chord, H chord, and three chord, respectively.  
%In this paper, 
%\section{Preliminary}\label{sec1.4}
%In this section, we fix our convention.  
The strong (resp.,~weak) second Reidemeister move, strong RI\!I (resp.,~weak RI\!I), is defined as the local replacement in Fig.~\ref{ych6}.  
\begin{figure}[h!]
\includegraphics[width=8cm]{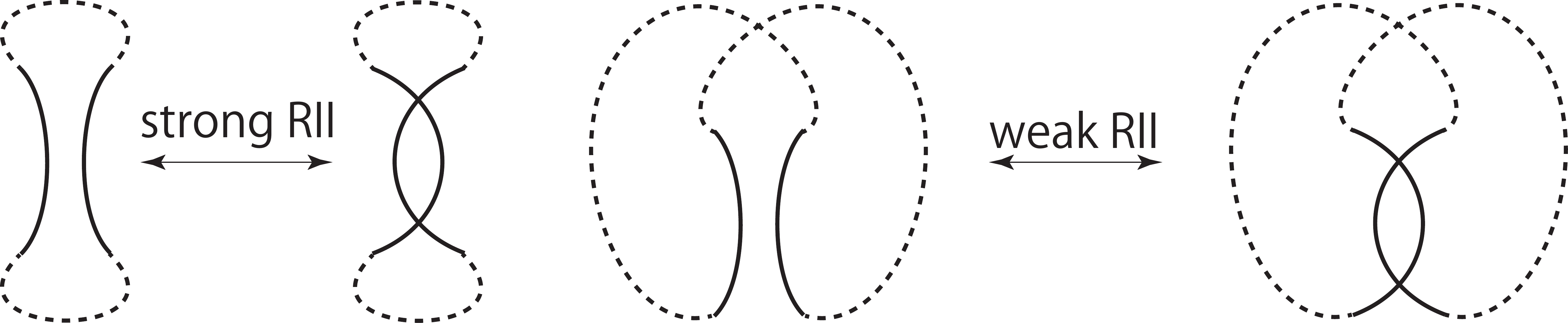}
\caption{Strong (left) and weak (right) second Reidemeister moves.  Dotted arcs indicate the connections of the branches.}\label{ych6}
\end{figure}
The strong (resp.,~weak) third Reidemeister move, strong RI\!I\!I (resp.,~weak RI\!I\!I), is defined as the local replacement in Fig.~\ref{ych7}.  
\begin{figure}[h!]
\includegraphics[width=8cm]{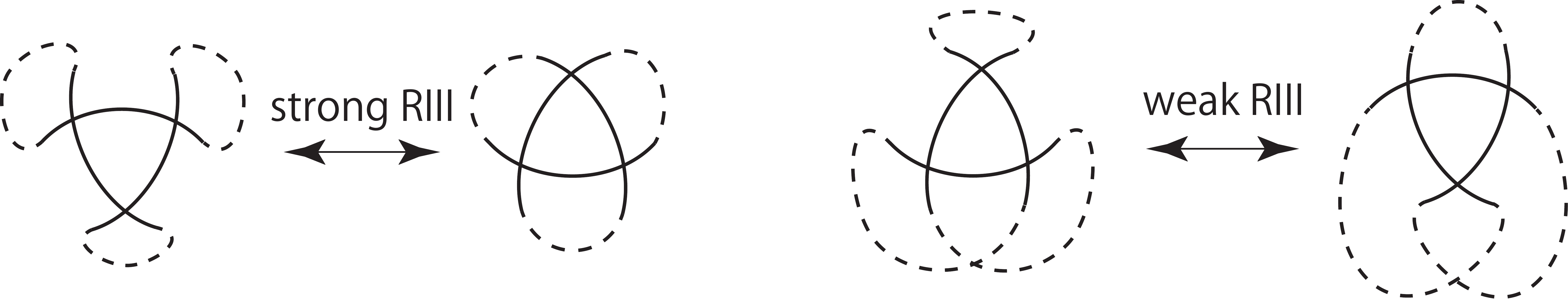}
\caption{Strong (left) and weak (right) third Reidemeister moves.  Dotted arcs indicate the connections of the branches.}\label{ych7}
\end{figure}
%We also discuss on series of sub-chords and a flype shown in Fig.~\ref{flyped}.  
%, which is famous by Tait Flyping Conjecture proved by \cite{kauffman, murasugi}.  
%\begin{figure}[h!]
%\includegraphics[width=5cm]{flyped.pdf}
%\caption{Flype.}\label{flyped}
%\end{figure}

Now we state some new results.  
%\section{Numbers of chords under Reidemeister moves.}
Theorem~\ref{main2_thm} gives a new integer-valued invariant $\lambda(P)$ that provides a complete classification of all prime reduced knot projections with up to seven double points under the equivalence relation induced by RI and strong RI\!I\!I (strong (1, 3) homotopy \cite{ITT}), as shown in Fig.~\ref{hyou5}.  
%\begin{theorem}\label{main2_thm}
Throughout this paper, let $x(P)$ be the number of sub-chord diagrams, each of which is a chord diagram $x=\otimes, \tr$, $\h$, $\HH$, or $\thr$ in $CD_P$ of an arbitrary knot projection $P$.  
\begin{theorem}\label{thm_sec4}
The increments or decrements under one first, weak second, strong second, weak third, and strong third Reidemeister moves are shown in Fig.~\ref{ych3}, respectively, where $m$ is an integer.  
\end{theorem}
\begin{figure}[h!]
\includegraphics[width=6cm]{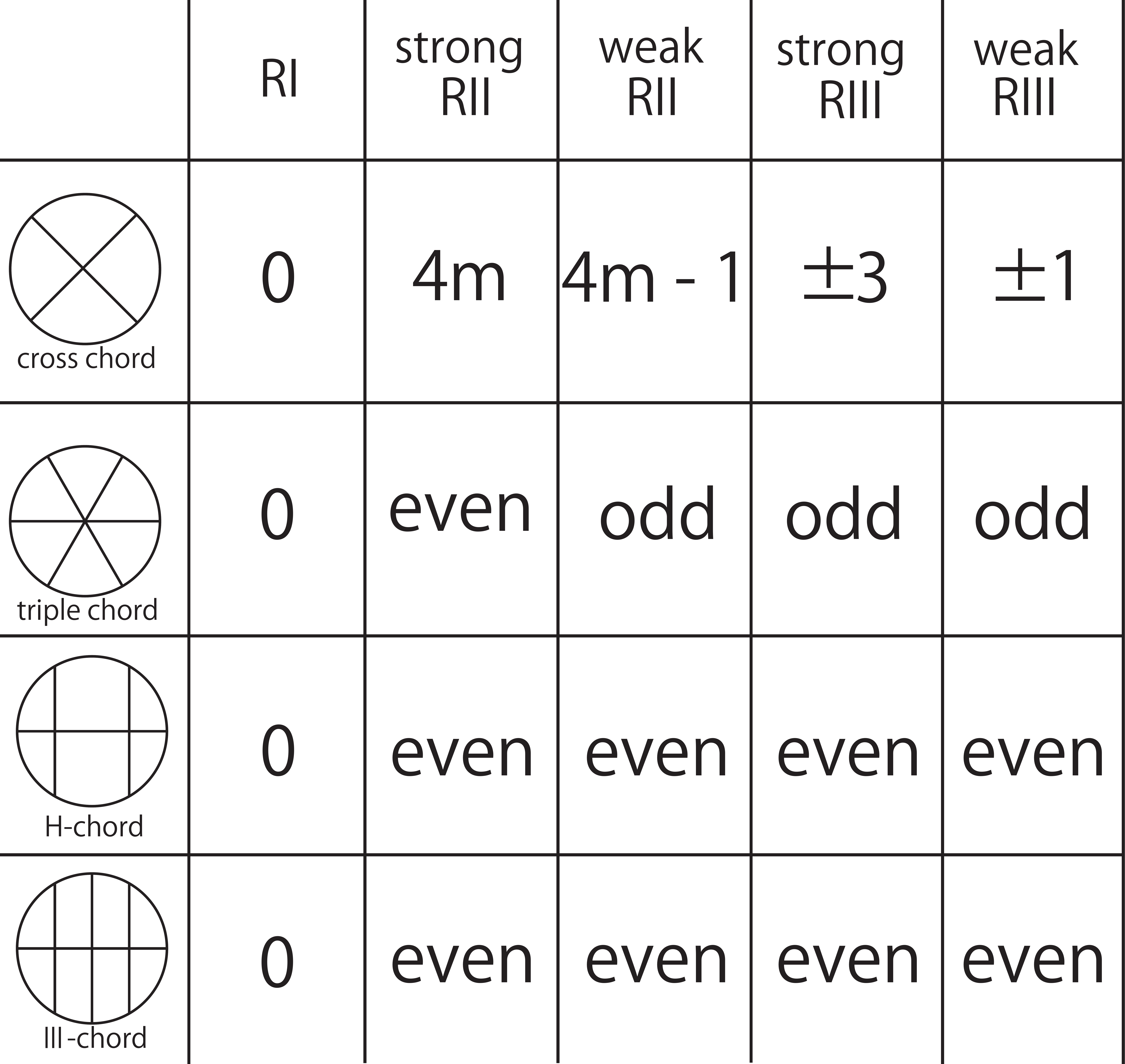}
\caption{Difference under several types of Reidemeister moves.  $m$ is an integer.}\label{ych3}
\end{figure}
Theorem \ref{main2_thm} is based on Theorem \ref{thm_sec4} and the discussion preceding it.
\begin{theorem}\label{main2_thm}
\[[3\h(P)- 3\tr(P)+\otimes(P)]/4\] is an integer $\lambda(P)$ that is invariant under RI and strong RI\!I\!I.   
\end{theorem}  
We remark that there has been no non-trivial integer-valued invariant, such as $\lambda(P)$, under RI and strong RI\!I\!I before.  The invariant $\lambda(P)$ is additive (Proposition~\ref{additive_prop}) and has the following important properties.  

%In the rest of the paper, the integer $[3\h(P)-3\tr(P)+\otimes(P)]/4$ is denoted by $\lambda(P)$.   
Below, using the invariant $H(P)$ defined in \cite{ITT} and introducing $\widetilde{X}(P)$, we have the following result, when $\lambda(P)=0$ (note that $\lambda(P)=0$ if $\widetilde{X}(P)=0$).  
\begin{theorem}\label{main3}
Let $P$ be an arbitrary knot projection.  
A map $H$ $($resp.,~$\widetilde{X}$$)$ from the set of all the knot projections to $\{0, 1\}$ is defined by the condition $H(P)=0$ $($resp.~$\widetilde{X}(P)=0$$)$ if and only if $\h(P)=0$ $($resp.,~$\otimes(P)=0$$)$.  Then $H(P)$ $($resp.,~$\widetilde{X}(P)$$)$ is invariant under RI and strong RI\!I\!I $($resp.,~weak RI\!I\!I$)$  %A simple closed curve on a sphere is denoted by $\bigcircle$.  
and we have the following four necessary and sufficient conditions:  
\begin{align*}
H(P)=0~{\text{and}}~\lambda(P)=0 &\Leftrightarrow {\text{$P$ can be related to $\bigcircle$ by using RI and strong RI\!I\!I,}} \\
\tr(P)=0~{\text{and}}~\lambda(P)=0 &\Leftrightarrow {\text{$P$ can be related to $\bigcircle$ by using RI,}} \\
\widetilde{X}(P)=0 &\Leftrightarrow {\text{$P$ can be related to $\bigcircle$ by using RI and weak RI\!I\!I,}}~{\text{and}}~ \\
\widetilde{X}(P)=0 &\Leftrightarrow {\text{$P$ can be related to $\bigcircle$ by using RI}} \\
\end{align*}
where $\bigcircle$ denotes a simple closed curve on a sphere.    
\end{theorem}

We also focus on sub-chord diagrams $\otimes$, $\tr$, $\h$, $\thr$, and $\HH$, each of which has an invariance property under fylpes, where a flype is a local replacement on knot projections as shown in Fig.~\ref{flyped}.
% defines a flype.  
\begin{figure}[h!]
\includegraphics[width=6cm]{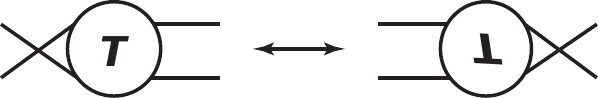}
\caption{Flype.}\label{flyped}
\end{figure}
%It is well known that 6/3\UTF{0082}±\UTF{0082}±\UTF{0082}\UTF{00DC}\UTF{0082}\UTF{00C5}%
\begin{theorem}\label{flype_thm}
$\otimes(P)$, $\tr(P)$, $\h(P)$, $\thr(P)$, and $\HH(P)$ are invariant under any flype.  
\end{theorem}
\begin{corollary}
$\lambda(P)$ is invariant under any flype.  
\end{corollary}
The remainder of this paper contains the following sections.  Theorem~\ref{thm_sec4}, \ref{main2_thm}, \ref{main3}, and \ref{flype_thm} are proved in Secs.~\ref{sec4}, \ref{sec_main2_thm}, \ref{sec3_main3} and \ref{sec_flype}, respectively.  Sec.~\ref{flype_thm} also mentions properties of $\lambda$, and Sec.~\ref{average_sec} comments on the behavior of Averaged invariant $-(J^{+} + 2 St)/2$ consisting of Arnold's invariants $J^{+}$ and $St$.  Finally, we present a table of prime reduced knot projections with up to seven double points, counting each type of sub-chord diagram embedded in a chord diagram of each knot projection.  

\section{Proof of Theorem \ref{thm_sec4}.}\label{sec4}
\begin{proof}
The sub-chord diagrams $\otimes$, $\tr$, $\h$, and $\thr$ are called a cross chord, a triple chord, an H-chord, and a I\!I\!I-chord, respectively.  Throughout this proof, $\Delta ch$ and $\Delta ch(x)$ are the difference of two numbers of embedded sub-chord diagrams and of embedded sub-chord diagrams specified by $x$, respectively, under a single local move that we focus on, where $x$ $=$ $\otimes$, $\tr$, $\h$, and $\thr$.  
\begin{enumerate}
\item RI (Fig.~\ref{ych16}).  Consider the first column of the table in Fig.~\ref{ych3}.  Move RI does not affect $\otimes(P)$, $\tr(P)$, $\h(P)$, or $\thr(P)$ for an arbitrary knot projection $P$.  
%the number of sub-chords: cross chords, triple chords, H chords, and three chords 
This implies that $\Delta ch = 0$ for every box in the first column.  
\begin{figure}
\includegraphics[width=5cm]{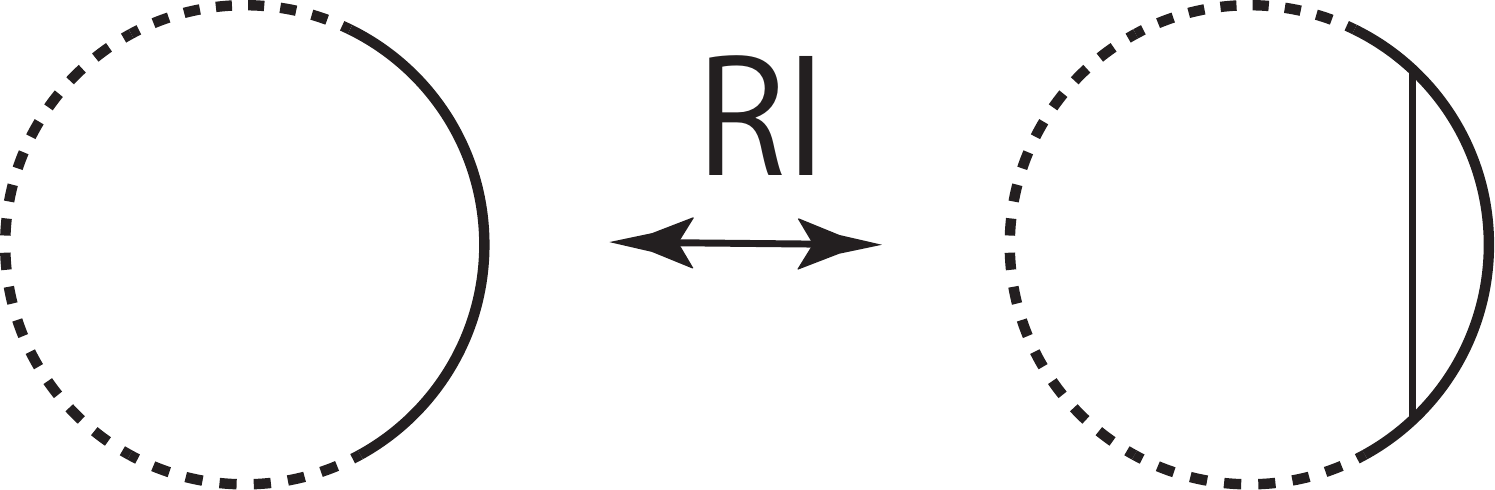}
\caption{RI on chord diagrams.}\label{ych16}
\end{figure}
\item Strong RI\!I (Fig.~\ref{ych17}).  Consider the second column of the table in Fig.~\ref{ych3}.  We separate the cases based on the number of chords that relate to a strong~RI\!I and belong to the new chord from the left to the right in Fig.~\ref{ych17}.  
\begin{figure}
\includegraphics[width=5cm]{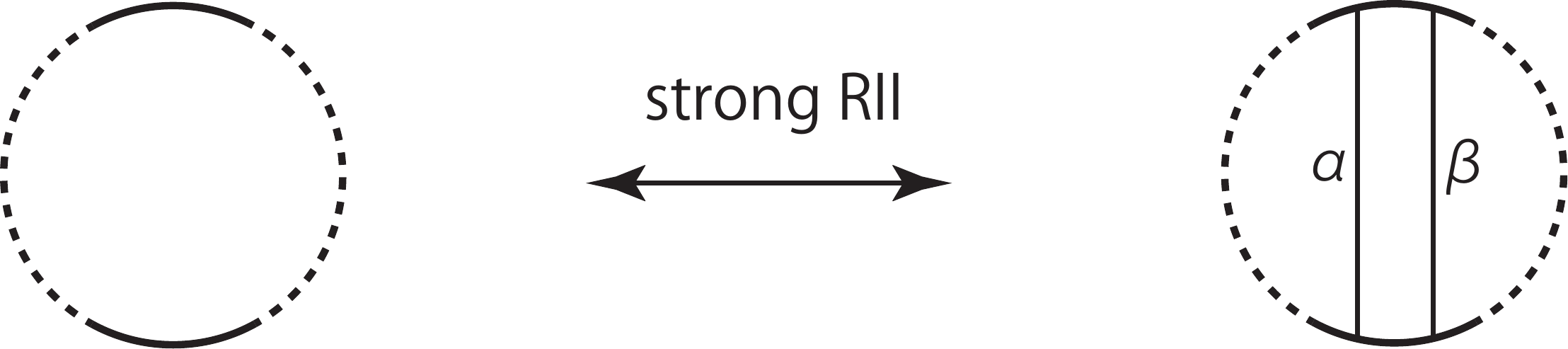}
\caption{Strong RI\!I on chord diagrams.}\label{ych17}
\end{figure}
\begin{itemize}
\item $\otimes$.  Any cross chord as a sub-chord cannot contain both chords $\alpha$ and $\beta$ in Fig.~\ref{ych17}.  Additionally, a chord crossing $\alpha$ (resp.,~$\beta$) should also cross $\beta$ (resp.,~$\alpha$).  We will call this type of the argument {\it{the duality $(\alpha, \beta)$}}.  Thus, the difference in $\otimes(P)$ of a knot projection $P$ by one strong RI\!I should be even.   Moreover, $\Delta ch(\otimes) =4m~(m \in \mathbb{Z})$.  The reason is as follows.  

See Fig.~\ref{ych18}.  For any knot projection $P$, if $P$ has $\alpha$ on the right of Fig.~\ref{ych18}, then the number of chords of $CD_P$ crossing $\alpha$ of $CD_P$ in the left figure of Fig.~\ref{ych18} is even.  This is because for two component spherical curves, if a component of an immersed spherical curve intersects another component, the number of double points consisting of two components is even (Fig.~\ref{ych19}).   
\begin{figure}[h!]
\includegraphics[width=5cm]{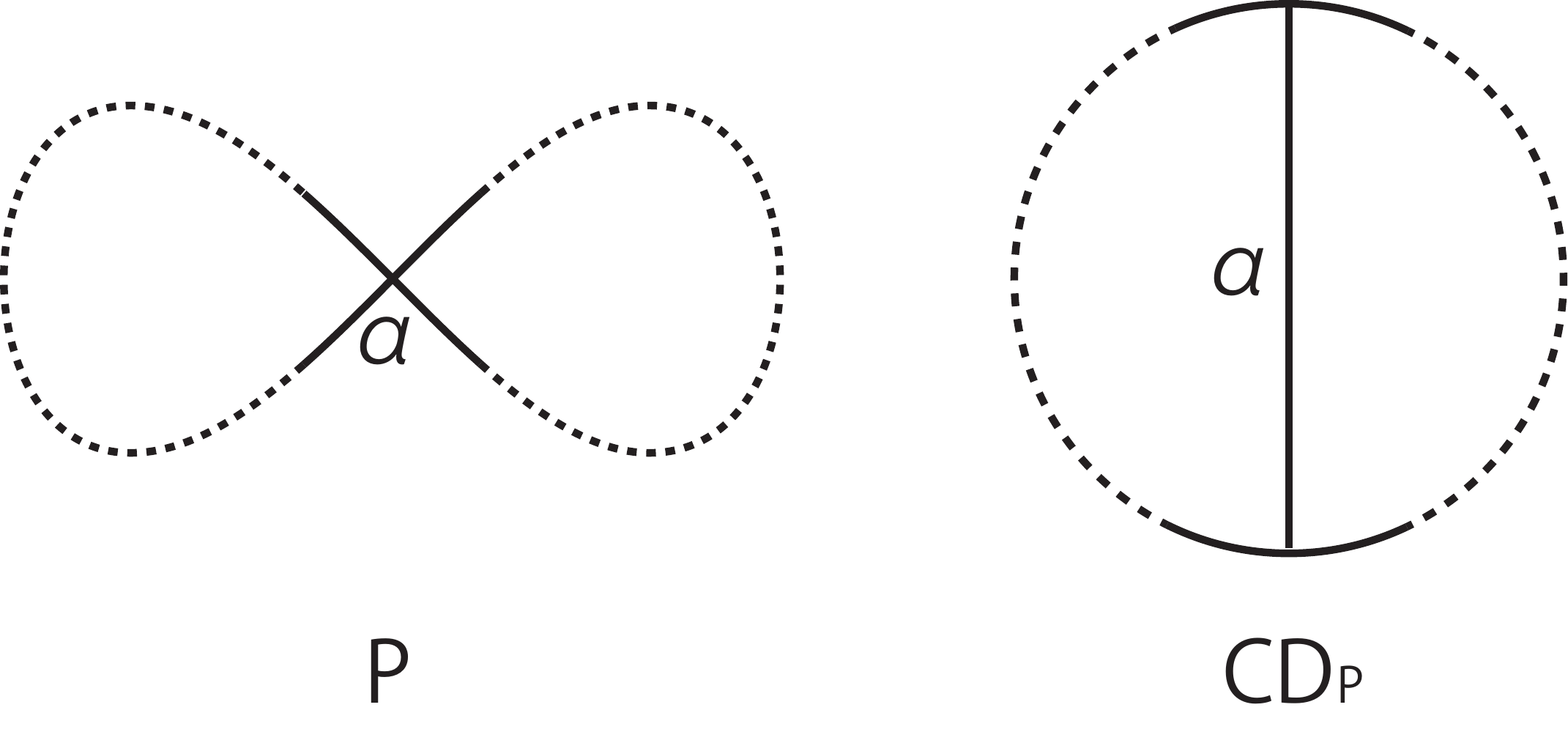}
\caption{Correspondence between a double point of a knot projection $P$ and a chord of chord diagram $CD_P$.}\label{ych18}
\end{figure}
\begin{figure}[h!]
\includegraphics[width=5cm]{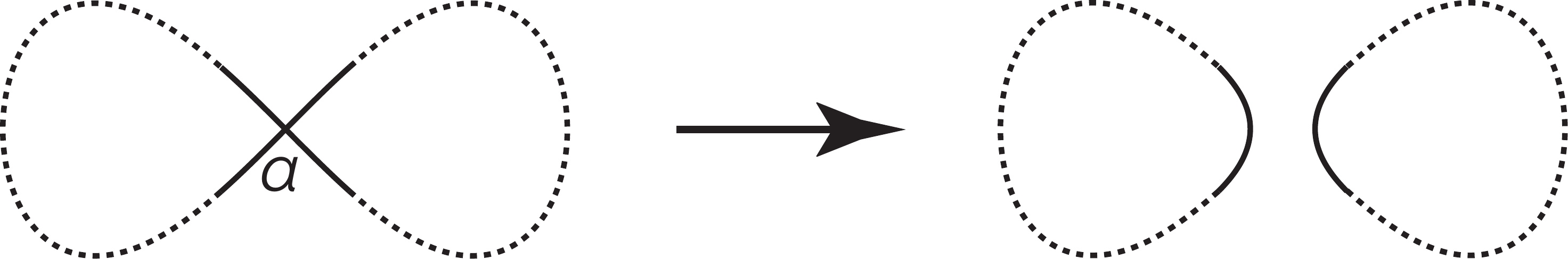}
\caption{Smoothing at a double point.  The number of intersections of two dotted knot projections is even.}\label{ych19}
\end{figure}
By referring either $\alpha$ or $\beta$ to $\alpha$ of Fig.~\ref{ych18}, $\Delta ch(\otimes)=4m~(m \in\mathbb{Z})$.  
%\item $\tr$. By an argument with respect to the duality $(\alpha, \beta)$, the difference of $\tr(P)$ of a knot projection $P$ should be even (i.e.,~$\Delta ch(\tr) ={\text{even}}$).  
%%%(\UTF{0090}V\UTF{0082}\UTF{00B5}\UTF{0082}¢\UTF{0095}\UTF{0094}\UTF{0095}\UTF{00AA}\UTF{0082}±\UTF{0082}±\UTF{0082}\UTF{00A9}\UTF{0082}\UTF{00E7})%%%%%%%%%%
\item $\tr$, $\h$, and $\thr$.  Since one strong RI\!I consists of two RIs, a strong RI\!I\!I, and a weak RI\!I, then $\Delta ch (\tr)=$ even, $\Delta ch(\h)=$ even, and $\Delta ch (\thr)=$ even by using results for RI, strong RI\!I\!I, and weak RI\!I. 
\end{itemize}

\item Weak RI\!I (Fig.~\ref{ych20}).  Consider the third column of the table in Fig.~\ref{ych3}.  
\begin{figure}[h!]
\includegraphics[width=5cm]{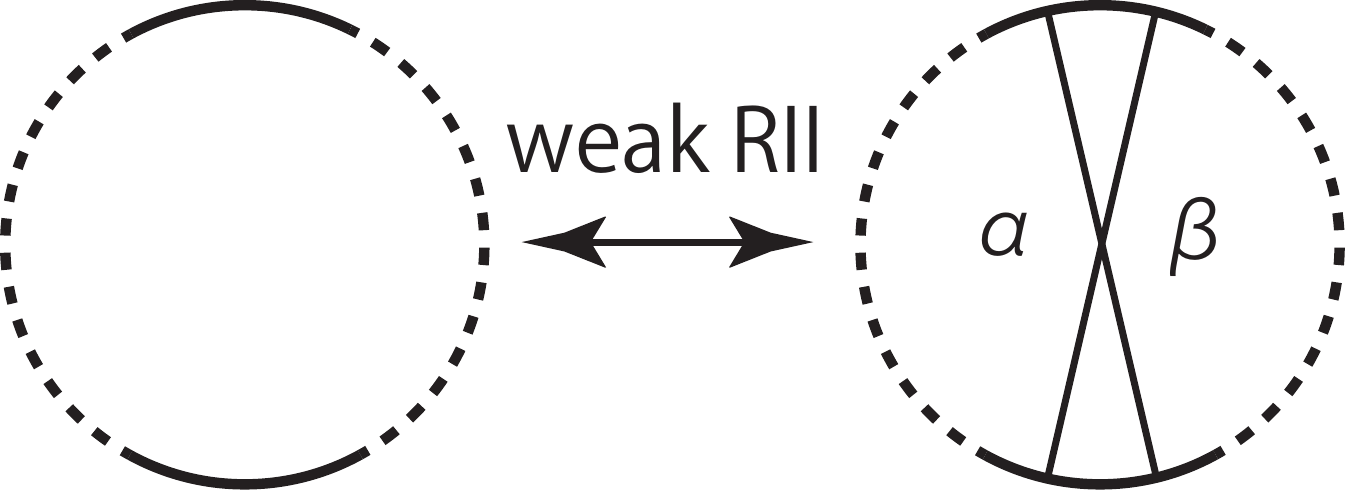}
\caption{Weak RI\!I on chord diagrams.}\label{ych20}
\end{figure}
\begin{itemize}
\item $\otimes$.  Let $\alpha$ and $\beta$ be chords specified in Fig.~\ref{ych20}.  If $\otimes$ consists of $\alpha$ (resp.,~$\beta$) and the other chord $e$ is not $\beta$ (resp.,~$\alpha$), there exists another $\otimes$ consisting of $\beta$ (resp.,~$\alpha$) and $e$ (the argument of the duality $(\alpha, \beta)$).  The number of the sub-chord $\otimes$ consisting of $\alpha$ and $\beta$ is one.  Then, $\Delta ch(\otimes)$ is odd.  Moreover, $\Delta ch (\otimes)= 2(2m-1) + 1$ $=$ $4m-1~(m \in \mathbb{Z}_{\ge 1})$ by the argument regarding Figs.~\ref{ych18} and \ref{ych19}.  
\item $\tr$.  We split the cases by how many increased chords belong to the new $\tr$ from the left to the right in Fig. ~\ref{ych20}.  
\begin{itemize}
\item (one chord in the new.) By the duality $(\alpha, \beta)$, the contribution to the difference is an even number of chords.  
\item (two chords in the new.)  
The number of chords such as $A$ shown in Fig.~\ref{ych20a} is odd 
%There exists a chord such as $A$ shown in Fig.~\ref{ych20a}, for which the number of such chords is odd 
by the argument regarding Figs.~\ref{ych18} and \ref{ych19} (cf., case $\otimes$ of strong RI\!I).  
\begin{figure}[h!]
\includegraphics[width=2cm]{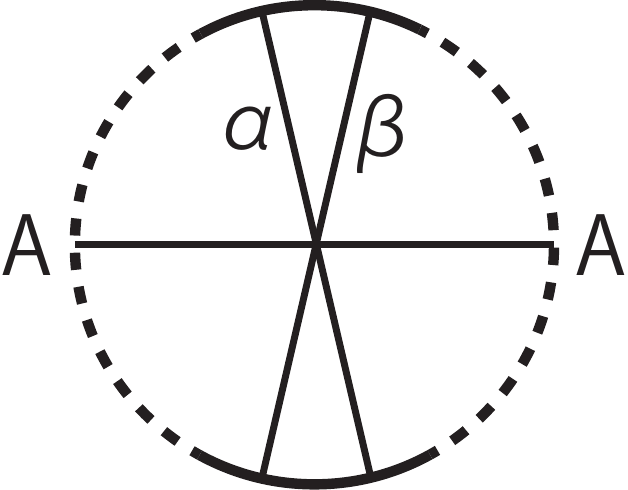}
\caption{Appearance of chord $A$ on the right of Fig.~\ref{ych20} under weak RI\!I.}\label{ych20a}
\end{figure}
\end{itemize}
\item $\h$.  There is no possibility that both $\alpha$ and $\beta$ are contained in the new $\h$ from the left to the right in Fig.~\ref{ych20}.  Then we consider the duality $(\alpha, \beta)$, which implies $\Delta ch(\h)=$ even.  
\item $\thr$.  This case is very similar to the above $\h$ case, $\Delta ch(\thr)=$ even.  
\end{itemize} 
\item Strong RI\!I\!I (Fig.~\ref{ych21}).  Consider the fourth column of the table in Fig.~\ref{ych3}.  
\begin{figure}[h!]
\includegraphics[width=5cm]{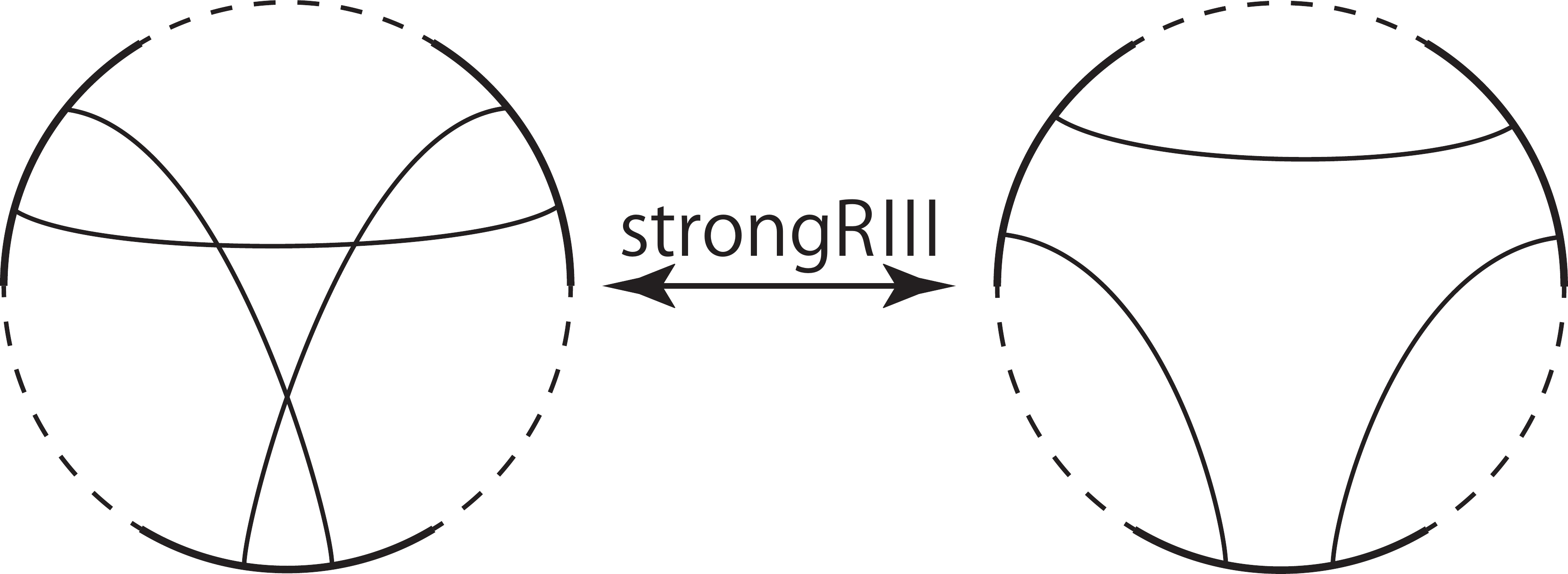}
\caption{Strong RI\!I\!I on chord diagrams.}\label{ych21}
\end{figure}
\begin{itemize}
\item $\otimes$.  Fig.~\ref{ych21} directly implies $\Delta ch(\otimes)=\pm 3$ (as in \cite{ITT}).  
\item $\tr$.  We split the cases by the number of chords that relate to the strong RI\!I\!I and belong to the new $\tr$ from right to left in Fig.~\ref{ych21}.  
\begin{itemize}
\item (one chord related to the new).  By Fig.~\ref{ych21}, $\Delta ch(\tr)=0$ in this case.  
%Just one $\tr$ belongs to the new $\tr$.  
\item (two chords related to the new).  
%The argument related to Fig.~\ref{ych18} implies that even chords belong to the new $\tr$.   
In this case, for each chord $X$ shown in Fig.~\ref{ych22}, the difference is one from left to right in Fig.~\ref{ych22} (using symmetry, it is sufficient to consider Fig.~\ref{ych22}).  
%Let us show that the difference is $\pm 1$ for each $X$ between the left and the right in Fig.~\ref{ych22}.  
\begin{figure}[h!]
\includegraphics[width=5cm]{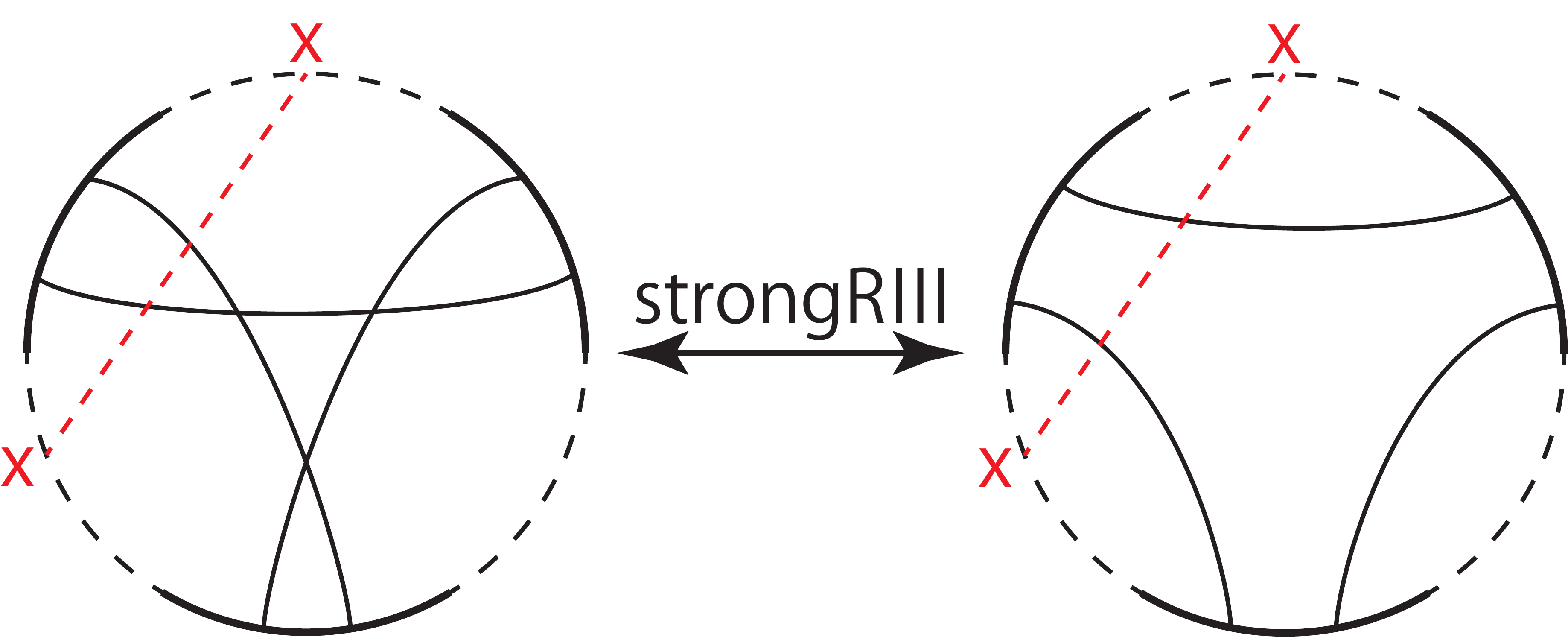}
\caption{Chord $X$ appearing in strong RI\!I\!I.}\label{ych22}
\end{figure}
%Let us count its difference.  For instance, when we focus on a red chord $X$, two black chords and $X$ on the left obtain $\tr$ but $X$ with any two of the three black chords does not make $\tr$.  
%In similar way, we can find another $\h$ containing chord $X$ in the left and one $\h$ in the right.  

%Then we can see the difference one $\tr$ from the left to the right for each $X$ as in Fig.~\ref{ych22}.  

%Thus we have the difference is $\pm 1$ for each $X$ under the move as shown in Fig.~\ref{ych22}.  

%\begin{figure}[h!]
%\includegraphics[width=5cm]{ych24.pdf}
%\end{figure}
We show that the number of such $X$ is even, as follows.  First, we apply the same argument as for Figs.~\ref{ych18} and \ref{ych19} to the leftmost figure of Fig.~\ref{ych7}.  The operation illustrated in Fig.~\ref{ych24a} is useful for showing this.  
\begin{figure}[h!]
\includegraphics[width=5cm]{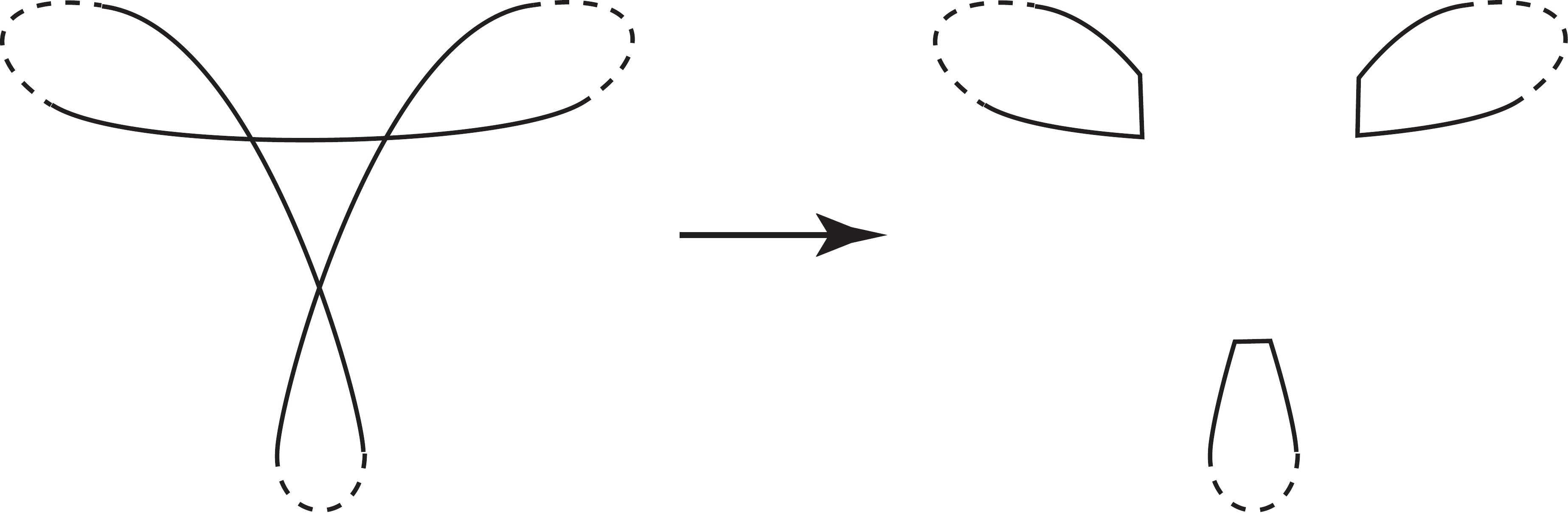}
\caption{Resolutions of three double points.}\label{ych24a}
\end{figure}
After the operation illustrated in Fig.~\ref{ych24a}, two of three components intersect at even double points, which implies the number of chords such as $X$ is even.  
%After the operation illustrated in Fig.~\ref{ych24a}, two of three components intersect at even double points 
%the operation as shown in Fig.~\ref{ych24}, 
%\begin{figure}[h!]
%\includegraphics[width=5cm]{ych24.pdf}
%\end{figure}
%Thus, we notice that the number of endpoints of chords for each of three dotted arcs is even, as shown in Fig.~\ref{ych23}.  
%\begin{figure}
%\includegraphics[width=2cm]{ych23.pdf}
%\caption{Each dotted arc in a chord diagram has even endpoints.}\label{ych23}
%\end{figure}
%For $Y$ $=$ $A$, $B$, or $C$, three dotted arcs within two $Y$'s in the circle of Fig.~\ref{ych23} is denoted by $(Y \sim Y)$.  Let $\sharp AB$ be the number of chords between $(A \sim A)$ and $(B \sim B)$.  Similarly, the number of chords between $(A \sim A)$ (resp.,~$(B \sim B)$) and $(C \sim C)$ is denoted by $\sharp AC$ (resp.,~$\sharp BC$).  Thus, we have two possibilities as follows.  
%\begin{itemize}
%\item $\sharp AB =$ even, $\sharp AC =$ even, and $\sharp BC =$ even.  In this case, the number of chords similar to $X$ is even.   
%\item $\sharp AB =$ odd, $\sharp AC =$ odd, $\sharp BC =$ odd.  
%In this case, reapplying the argument of Figs.~\ref{ych18} and \ref{ych19}, $\sharp AB$ must be even (Recall the argument of Fig.~\ref{ych24a}).  Thus, there is not possible.   
%\end{itemize}
%As a result, the number of $X$-type chords in Fig.~\ref{ych22} is even.  
Since each $X$ produces the difference $\pm 1$, the difference is even in this case.  
\item (three chords related to the new).  Only one $\tr$ belongs to the new $\tr$.  
%There is no new $\tr$ of such a case.  
%The increment and the decrement of the number of $\tr$ does not change.  
\end{itemize}
As a result, $\Delta ch(\tr)=$ odd.  
\item $\h$.  We separate the cases based on the number of chords,  in $\h$,  that related to the difference between the left and the right of Fig.~\ref{ych21}.  
\begin{itemize}
\item (one chord related to the difference) By Fig.~\ref{ych21}, $\Delta ch (\h)=0$.  
%In this case, there is no difference contributing $\Delta ch(\h)$.  
\item (two chords related to the difference)  The discussion is very similar to the case of $\tr$.  Consider Fig.~\ref{ych22}.  For each $X$, the difference in the number of H-chords is $\pm 1$ (from two H-chords (left) to one H-chord (right) in Fig.~\ref{ych22}).  The number of such $X$ is even by the fact that we showed in the case of $\tr$.  As a result, $\Delta ch(\h)=$ even.  
\item (three chords related to the difference) In this case, there is no difference contributing to $\Delta ch (\h)$ and $\Delta ch (\h)=0$.  
\end{itemize}
\item $\thr$.  We divide the cases by how many chords in $\thr$ relate to the difference under one strong RI\!I\!I.  
\begin{itemize}
\item (one chord related to the difference) In this case, the number of $I\!I\!I$-chords does not change under strong RI\!I\!I.  
\item (two chords related to the difference) Consider Fig.~\ref{ych24}.  
\begin{figure}[h!]
\includegraphics[width=5cm]{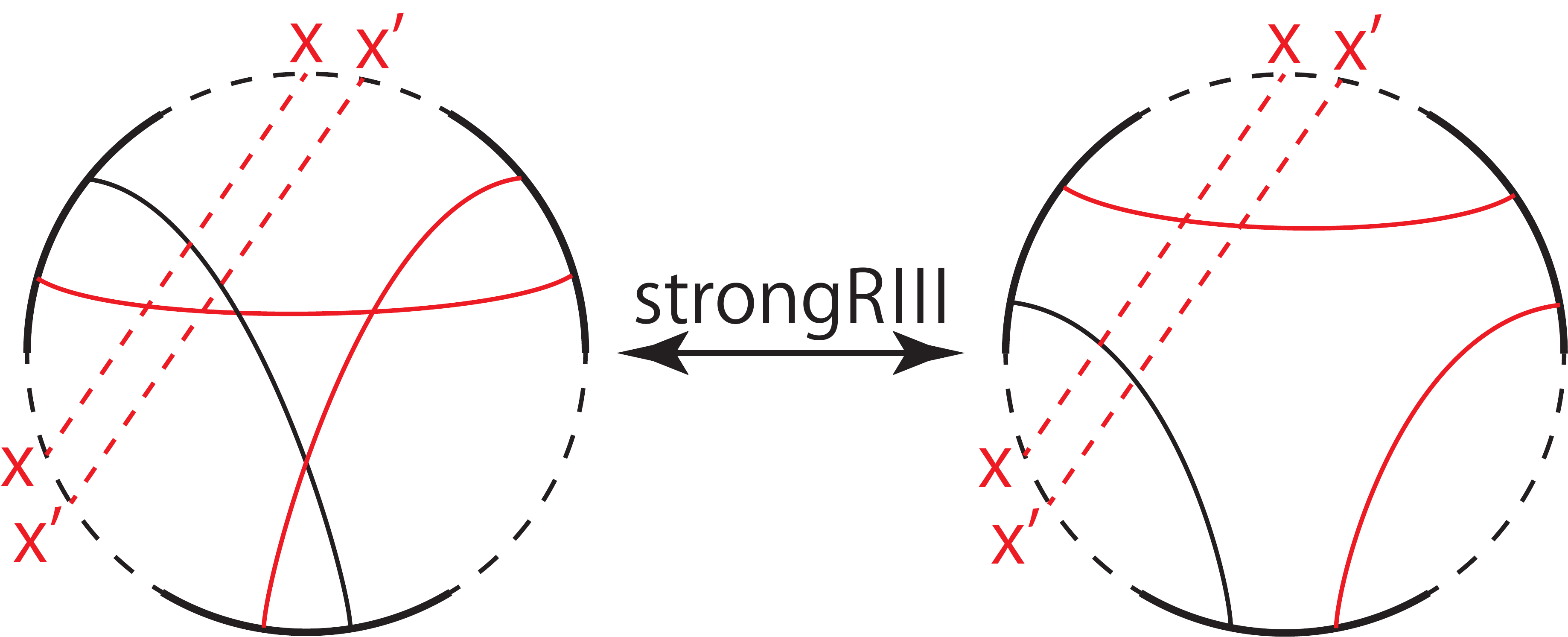}
\caption{Appearance of two parallel chords $X$ and $X'$ under strong RI\!I\!I.  Two I\!I\!I-chords, both including $X$ and $X'$, are contained on the left hand side.}\label{ych24}
\end{figure}
As in Fig.~\ref{ych24}, the difference in the number of $I\!I\!I$-chords is $\pm 2$ for each pair $(X, X')$.  
%Red chords of Fig.~\ref{ych24}, one $\thr$ is changed.  Another $\thr$ with $X$ and $X'$ is changed from the left to the right on Fig.~\ref{ych24}.  Thus, for each pair $(X, X')$, the difference of the number of $\thr$ in this cases is $\pm 2$.  
Therefore, for all pairs matching $(X, X')$, the difference of the sum is even.  

We consider another type of possibility, as shown in Fig.~\ref{ych30}.  
\begin{figure}[h!]
\includegraphics[width=5cm]{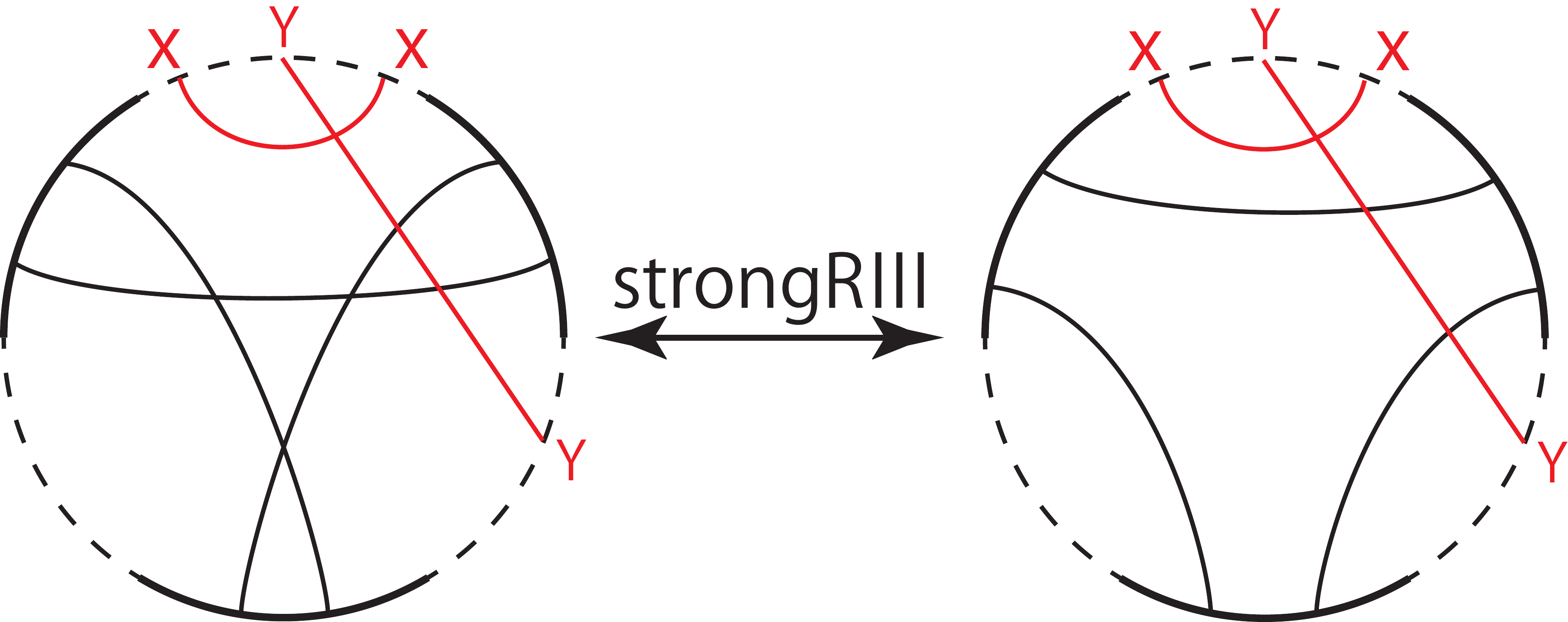}
\caption{Appearance of $X$-type and $Y$-type chords under strong RI\!I\!I.  A I\!I\!I-chord including $X$ and $Y$ is contained on the right hand side.}\label{ych30}
\end{figure}
For each chord $X$, using Figs.~\ref{ych18} and \ref{ych19} that is frequently used above, the number of $Y$-type chords shown in Fig.~\ref{ych30} is even.  
%%(\UTF{0090}\UTF{00E0}\UTF{0096}\UTF{0178})%%
The detailed explanation is as follows.  See Fig.~\ref{abc}.  
First, apply the operation shown in Fig.~\ref{ych24a} to the considered diagrams.  Second, select the curve $C_1$ containing the chord $X$.  Third, in the curve $C_1$, apply the operation shown in Fig.~\ref{ych19} to the double point corresponding to $X$.  Now we have four component curves, of which, two curves intersect at even double points.  This is why the number of $Y$-type chords is even.  
\begin{figure}[h!]
\includegraphics[width=8cm]{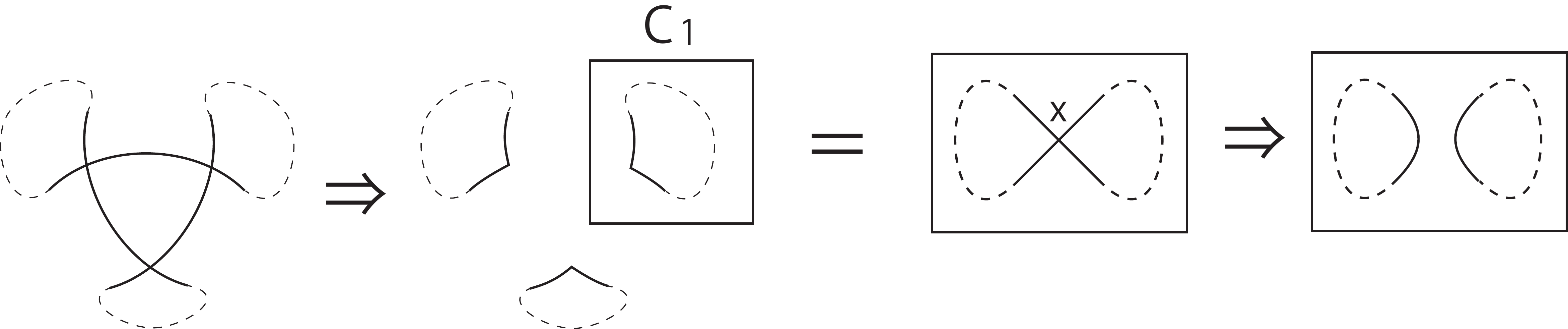}
\caption{The number of $Y$-type chords is even.}\label{abc}
\end{figure}
%%(\UTF{0090}\UTF{00E0}\UTF{0096}\UTF{0178})%%
Then, for each such $X$, the difference is even.  Therefore, the difference of the sum $\Delta ch (\thr)$ is even.  
\item (three chords related to the difference) There is no possible case.  
\end{itemize}
As a result, $\Delta ch(\thr)=$ even.  
\end{itemize}
\item Weak RI\!I\!I (Fig.~\ref{ych25}).  Finally, count the difference in $\Delta ch$ under one weak RI\!I\!I.  
\begin{figure}[h!]
\includegraphics[width=5cm]{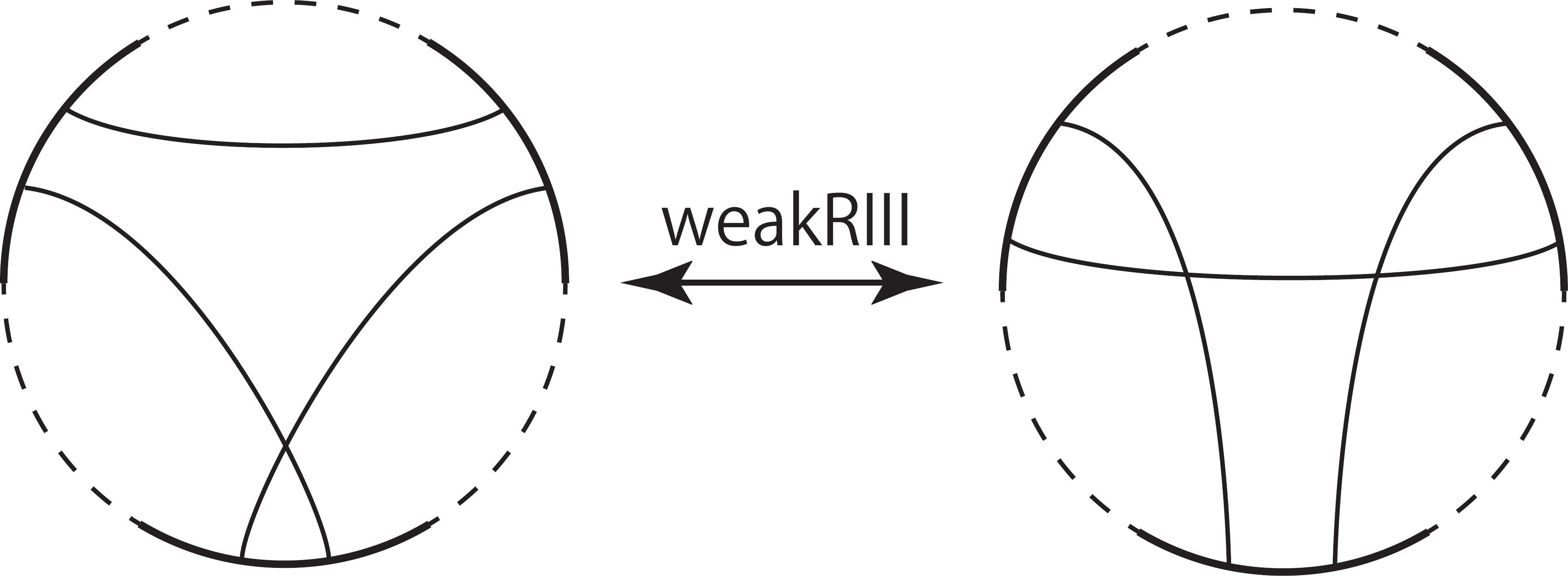}
\caption{Weak RI\!I\!I on chord diagrams.}\label{ych25}
\end{figure}
\begin{itemize}
\item $\otimes$.  From Fig.~\ref{ych25}, $\Delta ch(\otimes) = \pm 1$.  
\item $\tr$, $\h$, and $\thr$.  Since one weak RI\!I\!I consists of two strong RI\!Is and a strong RI\!I\!I, $\Delta ch(\tr)=$ odd, $\Delta ch(\h)=$ even, and $\Delta ch(\thr)=$ even.  
\end{itemize} 
\end{enumerate}
\end{proof}
\begin{remark}
\cite{ITT} contains the results for the first, strong third, and weak third Reidemeister moves on $\otimes$.  
\end{remark}
%\begin{remark}
%\cite{ITT} introduced strong (1, 3) invariant $H(P)$ for an arbitrary knot projection $P$.  The definition of $H(P)$ is obtained by the chord diagrams; the chord diagram $CD_P$ contains (resp.~does not contain) $\h$ if and only if $H(P)=1$ (resp.~$H(P)=0$).  The checking of the invariance is easy by using Figs.~\ref{ych16} and \ref{ych21}.  
%\end{remark}
%\begin{remark}
%Let $P$ be an arbitrary knot projection.  Using the number of $\otimes$ introduced the cross chord number $X(P)$ in \cite{ITT}, we can show \cite[Corollary 4.1]{IT} easily as follows.  
%\begin{proof}
%Set
%\begin{equation}
%\widetilde{X}(P) = \begin{cases} 0 \quad {\text{if}}~X(P)=0 \\ 1 \quad {\text{if}}~X(P) \neq 0 \end{cases}.  
%\end{equation}
%If $P$ is related to $\bigcircle$ by a finite sequence consisting of the first Reidemeister moves and the weak second Reidemeister moves, $\widetilde{X}(P)=0$.  $\widetilde{X}(P)=0$ if and only if $X(P)=0$.  If $X(P)=0$, $CD_P$ has no cross chord and then $P$ is related to $\bigcircle$ by a finite sequence of the first Reidemeister moves.  Conversely, if $P$ is related to $\bigcircle$ by a finite sequence of the first Reidemeister moves, $P$ is related by a finite sequence of the first and the weak second Reidemeister moves.  
%\end{proof}
%\end{remark}
\begin{remark}
Using Theorem \ref{thm_sec4}, we can easily create some invariants of knot projections in simple ways.  
%For example, let $X(P)$ (resp.,~$\tr(P)$) be the number of $\otimes$-type (resp.,~$\tr$-type) sub-chord diagrams in $CD_P$.  Similarly, let $\h(P)$ (resp.,~$\thr(P)$) be the number of $\h$-type (resp.,~$\thr$-type) sub-chord diagrams in $CD_P$.  
Observing the table in Fig.~\ref{ych3}, $\otimes(P) \equiv \tr(P)$ (mod $2$) and we notice that $\otimes(P)$ (mod $2$), equivalently $\tr(P)$ (mod $2$), is invariant under the first and strong second moves.  $\otimes(P)$ (mod $3$) is invariant under the first and strong third moves, introduced in \cite{ITT}.  $\otimes(P)$, $\tr(P)$, $\h(P)$, and $\thr(P)$ are invariants under the first move.  
\end{remark}
\section{Proof of Theorem \ref{main2_thm}.}\label{sec_main2_thm}
\begin{proof}
To show Theorem \ref{main2_thm}, we check the difference of \[3\h(P) - 3\tr(P) + \otimes(P)\] under Reidemeister moves RI, strong RI\!I\!I, strong RI\!I, weak RI\!I, and weak RI\!I\!I in that order.  
\begin{itemize}
\item RI.  There is no change of $\otimes(P)$ under RI, and then neither $\h(P)$ nor $\tr(P)$ changes under RI.  
\item Strong RI\!I\!I.  Consider Fig.~\ref{ych22}.  Four chords consisting of a dotted chord $X$, called an {\it{$X$-type chord}}, and three other chords, called {\it{RI\!I\!I chords}}, are depicted explicitly.  In addition, recall that the difference in $\otimes(P)$ of a knot projection $P$ under strong RI\!I\!I is exactly $\pm 3$ supplied by three RI\!I\!I chords.  
\begin{itemize}
%\item Difference of contributions by three $X$-type chords.  
\item Difference of contributions by two non-RI\!I\!I chords and one RI\!I\!I chord.  There are no changes with respect to $\otimes(P)$, $\h(P)$, and $\tr(P)$, respectively.  
\item Difference of contributions by one non-RI\!I\!I chord and two RI\!I\!I chords.  
It is sufficient to consider the difference of contributions by one $X$-type chord and two RI\!I\!I chords.  
\begin{center}
\begin{tabular}{|c|c|} \hline
$x$ & $x$(left) $-$ $x$(right) in Fig.~\ref{ych22} \\ \hline
$\h(P)$ & $1$ \\ \hline
$\tr(P)$ & $1$ \\ \hline
$\otimes(P)$ & $0$ \\ \hline
\end{tabular}
\end{center}
Thus, there is no difference of $3 \h(P) - 3 \tr(P) + \otimes(P)$ in this case.  
\item Difference of contributions by three RI\!I\!I chords.  
\begin{center}
\begin{tabular}{|c|c|} \hline
$x$ & $x$(left) $-$ $x$(right) in Fig.~\ref{ych22} \\ \hline
$\h(P)$ & $0$ \\ \hline
$\tr(P)$ & $1$ \\ \hline
$\otimes(P)$ & $3$ \\ \hline
\end{tabular}
\end{center}
Thus, there is no difference of $3\h(P) - 3\tr(P) + \otimes(P)$ in this case.  
\end{itemize}
\item Strong RI\!I.  Consider Fig.~\ref{z1}.  
\begin{figure}[h!]
\includegraphics[width=5cm]{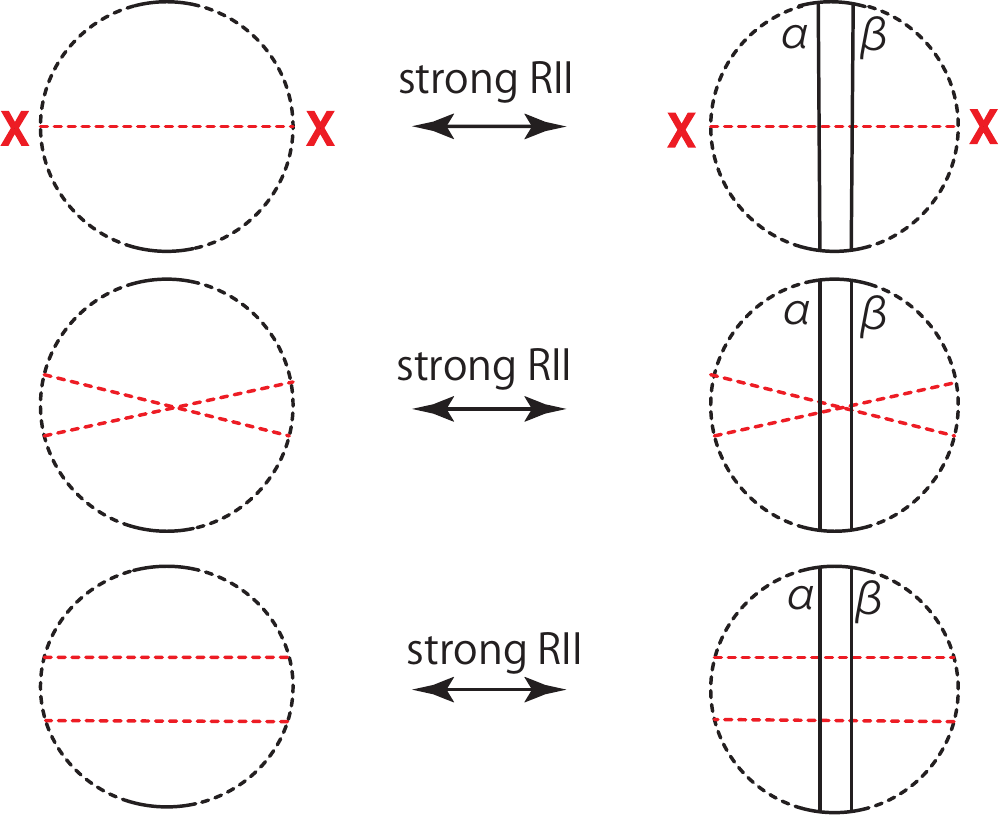}
\caption{Dotted chord in the top line: sticking chord $X$, dotted chords in the middle line: a pair of cross chords, dotted chords in the bottom line: a pair of parallel chords.}\label{z1}
\end{figure}
\begin{itemize}
\item (Top line of Fig.~\ref{z1}) We can assume that the number of chords crossing both $\alpha$ and $\beta$, called {\it{sticking chords}}, is $2m$ ($m \in \mathbb{Z}_{\ge 0}$) by the discussion with respect to Figs.~\ref{ych18} and \ref{ych19}.  

First, we count the difference caused by the tuple, each of which consists of one sticking chord, $\alpha$, and $\beta$.  
%Note that for $\otimes$, we can consider this case, i.e., it is sufficient to consider the difference caused by every sticking chord.
\begin{center}
\begin{tabular}{|c|c|} \hline
$x$ & $x$(right)-$x$(left) in Fig.~\ref{z1} \\ \hline
$\h(P)$ & $2m$ \\ \hline
$\tr(P)$ & $0$ \\ \hline
$\otimes(P)$ & $4m$ \\ \hline
\end{tabular}
\end{center}
\item (Center line of Fig.~\ref{z1})
Count the difference with respect to tuples, each of which consists of either $\alpha$ or $\beta$ and two sticking chords that mutually intersect and respectively intersect both $\alpha$ and $\beta$ (Fig.~\ref{z1}, the center line).  
%Consider $2m$ sticking chords.  
Assume that such $l$ pairs in all the $\binom{2m}{2}$ pairs are types as the center of Fig.~\ref{z1}.  
\begin{center}
\begin{tabular}{|c|c|}\hline
$x$ & $x$(right) $-$ $x$(left) in Fig.~\ref{z1} \\ \hline
$\h(P)$ & $0$  \\ \hline
$\tr(P)$ & $2l$  \\ \hline
$\otimes(P)$ & $0$ \\ \hline
\end{tabular}
\end{center}
\item (Bottom line of Fig.~\ref{z1})  Consider $(\binom{2m}{2} - l)$ pairs, as in the bottom line of Fig.~\ref{z1}.  
\begin{center}
\begin{tabular}{|c|c|} \hline
$x$ & $x$(right) $-$ $x$(left) in Fig.~\ref{z1} \\ \hline
$\h(P)$ &  $2 (\binom{2m}{2} - l)$ \\ \hline
$\tr(P)$ &  $0$ \\ \hline
$\otimes(P)$ & $0$ \\ \hline
\end{tabular}
\end{center}
\end{itemize}

Therefore, the difference is \[ 3 \{ 2(\binom{2m}{2}-l) + 2m \} - 3 \cdot 2l + 4m = 12m^2 +4m -12l.\]
\item Weak RI\!I.  
Since weak RI\!I consists of two RIs, a strong RI\!I\!I, and three strong RI\!Is, the difference of $3 \h(P)$ $-$ $3 \tr(P)$ $+$ $\otimes(P)$ is $4n$ ($n \in \mathbb{Z}_{\ge 0}$) under one weak RI\!I.  
\item Weak RI\!I\!I.  Since one weak RI\!I\!I consists of two strong RI\!Is and a strong RI\!I\!I, the difference of $3\h(P)-3\tr(P)+\otimes(P)$ is $4n~(n \in \mathbb{Z}_{\ge 0})$ under one weak RI\!I\!I.  
\end{itemize}

Any knot projection $P$ is related to a simple closed curve $\bigcircle$ by a finite sequence consisting of RI, RI\!I, and RI\!I\!I.  RI\!I (resp.,~RI\!I\!I) consists of strong and weak RI\!I (resp.,~RI\!I\!I).  Now we have that each difference of RI, RI\!I, and RI\!I\!I is $4 k~(k \in \mathbb{Z})$ and $3\h(\bigcircle)-3\tr(\bigcircle)+\otimes(\bigcircle)=0$.  The conditions complete the proof.  
\end{proof}
\section{Proof of Theorem~\ref{main3} and properties of $\lambda$.}\label{sec3_main3}
\begin{proof}
First, we recall that $H(P)$, (resp.,~$\widetilde{X}(P)$) is invariant under RI and strong RI\!I\!I (resp.,~weak RI\!I\!I).  
\begin{itemize}
\item {\it{Invariance of $H(P)$}} (cf.~\cite{ITT}).  By Theorem \ref{thm_sec4}, $H(P)$ is invariant under RI.  
%We consider every possible $\h$ between two chord diagrams, as in Fig.~\ref{ych21}, with each chord diagram in Fig.~\ref{ych21} representing three chords, called {\it{RI\!I\!I chords}}, explicitly as above.  We recall $X$-type chords, as in Fig.~\ref{ych22}.  
%\begin{itemize}
%\item Case 1: $\h$ consisting of three RI\!I\!I chords.  There is not a possible case.  
%\item Case 2: $\h$ consisting of two RI\!I\!I chords and a non-RI\!I\!I chord.  This non-black chord should be an $X$-type chord.  Fig.~\ref{ych22} shows that $H(P)=1$ holds under strong RI\!I\!I.  
%\item Case 3: $\h$ consisting of one RI\!I\!I chord and two non-RI\!I\!I chords.  One of the non-RI\!I\!I chords is at least an $X$-type chord.  When we add one arbitrary chord to Fig.~\ref{ych22} under this condition, the equation $H(P)=1$ holds under strong RI\!I\!I.    
%\item Case 4: $\h$ consisting of three non-RI\!I\!I chords.  In this case, $H(P)=1$ holds under strong RI\!I\!I, because no black chords does not affect such an $\h$.   
%\end{itemize}
%Therefore, any cases satisfies $H(P)=1$ under strong RI\!I\!I if $CD_P$ contains $\h$.  

Now we assume that $H(P)=0$ on the right of Fig.~\ref{ych22}.  In this case, no $X$-type chord can appear in Fig.~\ref{ych22}.  Then, no RI\!I\!I chords can be involved with producing $\h$ on the right of Fig.~\ref{ych21}.  If three non-RI\!I\!I chords comprise $\h$ in the right chord diagram in Fig.~\ref{ych21}, the right chord diagram has $\h$, which is contradicts $H(P)=0$.  Thus, we notice that $CD_P$ has no $\h$ in the right of Fig.~\ref{ych21} if and only if there is no $X$-type chord and no tuple of three non-RI\!I\!I chords comprising any $\h$.  We can also say that $CD_P$ has no $\h$ on the left of Fig.~\ref{ych21} if and only if there is no $X$-type chord and no tuple of three non-RI\!I\!I chords comprising any $\h$.  Therefore, when we denote the left (resp.,~right) knot projection by $P_l$ (resp.,~$P_r$) of the arrow of strong RI\!I\!I in Fig.~\ref{ych7}, 
\[H(P_l)=0 \Leftrightarrow H(P_r)=0.\]
Thus, \[H(P_l)=1 \Leftrightarrow H(P_r)=1.\]
%$H(P)=0$ holds under strong RI\!I\!I.  
\item {\it{Invariance of $\widetilde{X}(P)$.}}  By Theorem \ref{thm_sec4}, $\widetilde{X}(P)$ is invariant under RI.  Every chord diagram appearing in Fig.~\ref{ych25} always satisfies $\widetilde{X}(P)=1$.  
\end{itemize}
Second, we recall one of the facts from Sakamoto-Taniyama \cite[Theorem 3.2]{ST}.  
\begin{fact}[Sakamoto-Taniyama \cite{ST}]\label{st_thm}
Let $P$ be an immersed plane curve.  
A chord diagram $CD_P$ does not contain \h if and only if $P$ is equivalent to any connected sum of some plane curves, each of which is equivalent to one of the plane curves as $\bigcircle$, $\infty$, $P_1$, $P_2$, $P_3$, \dots as illustrated in Fig.~\ref{torus_knots}.  
\begin{figure}[h!]
\includegraphics[width=8cm]{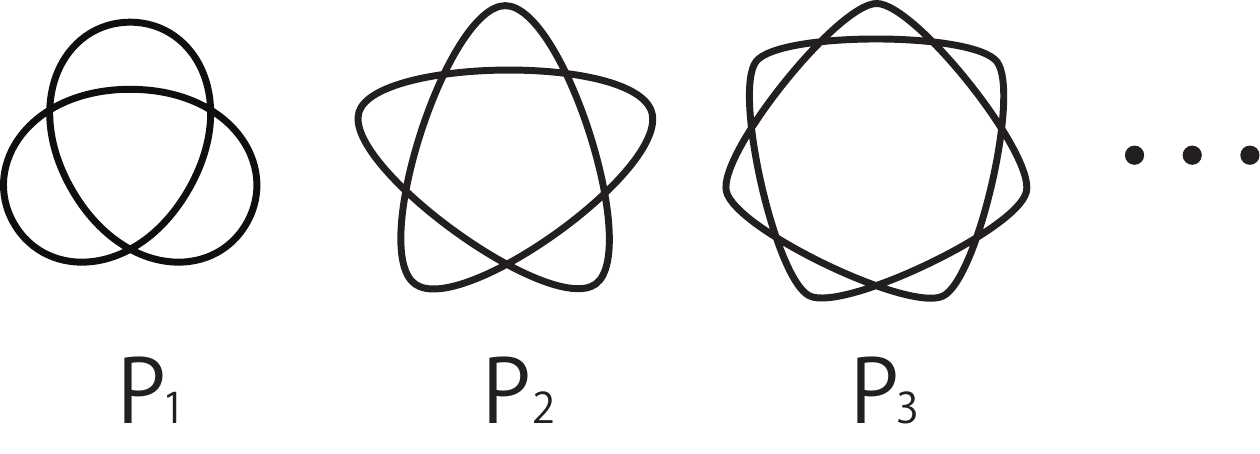}
\caption{$(2, 2i+1)$-torus knot projection $P_{i}$.}\label{torus_knots}
\end{figure}
\end{fact}
Notice that the same claim holds for knot projections with any possible choice of exterior region.  
%\begin{fact}[Corollary of Fact \ref{st_thm}]
\begin{align*}
H(P)=0 \Leftrightarrow & P~{\text{is any connected sum of knot projections, each of which is}} \\
&{\text{an element of}}~\mathcal{T}=\{ \bigcircle, \infty, {P_i}~(1 \le i \in \mathbb{Z})\}.  
\end{align*}
%\end{fact}
Now, we prove the first formula of Theorem~\ref{main3}.  Assume that $H(P)=0$ and $\lambda(P)=0$.  In this case, it is sufficient to consider any connected sum of elements in $\mathcal{T}$.  Note that $\lambda(P \sharp P')$ $=$ $\lambda(P)$ $+$ $\lambda(P')$, where $P \sharp P'$ is the connected sum of $P$ and $P'$.  
%If $P_{2i+1}$ holds $3\tr(P)=\otimes(P)$, then $3 \binom{2i+1}{3}=\binom{2i+1}{2}$, which implies $i=1$.  
Thus $\lambda(P) < 0$ if $P$ is a connected sum of knot projections satisfying $H(P)=0$ and at least one member is $P_{2i+1}~(i > 1)$.  This is because $\lambda(\bigcircle)=\lambda(\infty)=0$ and $4 \lambda(P_{2i+1})$ $=$ $-3\tr(P_{2i+1})+\otimes(P_{2i+1})$ $=$ $-3 \binom{2i+1}{3}$ $+$ $\binom{2i+1}{2}$ $=$ $i(2i+1)(1-i)$.  
%If $P$ is an arbitrary connected sum of knot projections in $\mathcal{T}_1$ $=$ $\{ \bigcircle, \infty, P_1\}$, $H(P)=\lambda(P)=0$.  
Then, if $H(P)=\lambda(P)=0$, $P$ is a connected sum of knot projections, each of which is an element of ${\mathcal{T}}_1$ $=$ $\{\bigcircle, \infty, P_1\}$.  Therefore, $P$ can be related to a simple closed curve $\bigcircle$ by a finite sequence consisting of RI and strong RI\!I\!I. 

Conversely, if $P$ can be related to a simple closed curve $\bigcircle$ by a finite sequence consisting of RI and strong RI\!I\!I, then $H(P)=\lambda(P)=0$.  
 
%%(omit)%%%%
%Moreover, we recall the following Fact~\ref{itt_thm} from \cite{ITT}.  
%\begin{fact}[Ito-Takimura-Taniyama \cite{ITT}]\label{itt_thm}
%The following (\ref{1}) and (\ref{2}) mutually equivalent.  
%\begin{enumerate}
%\newcommand{\enumi}{(\arabic{1})}
%\item A knot projection $P$ is any connected sum of knot projections, each of which is an element of $\mathcal{T}_1$.  \label{1}
%\item A knot projection $P$ and a simple closed curve $\bigcircle$ on the sphere can be related by a finite sequence of RI and strong RI\!I\!I.  \label{2}
%\end{enumerate}
%\end{fact}
%%(omit)%%
Then, we have
\[H(P)=\lambda(P)=0 \Leftrightarrow P~{\text{can be related to $\bigcircle$ by using $RI$ and strong RI\!I\!I}.}\]
This completes the proof of the first formula of Theorem~\ref{main3}.  

Next, we show the fourth formula before considering the second and third formulae.  Assume that $\widetilde{X}(P)=0$.  In this case, we have a chord diagram with no chord intersections. A knot projection $P$ having such a chord diagram can be related to a simple closed curve $\bigcircle$ by a finite sequence consisting of RI.  Conversely, if a knot projection $P$ can be related to $\bigcircle$ by a finite sequence consisting of RI, we have $\otimes(P)=0$, which implies $\widetilde{X}(P)=0$.  Then, we have the fourth formula
\begin{equation*}\label{eq1}
\widetilde{X}(P)=0 \Leftrightarrow P~{\text{can be related to $\bigcircle$ by using $RI$}.}
\end{equation*}
Now, we consider the third formula.  Since $\widetilde{X}(P)$ is invariant under RI and weak RI\!I\!I, 
%Now, assume that $P$ can be related to $\bigcircle$ by a finite sequence consisting of RI and weak RI\!I\!I.  
%If weak RI\!I\!I is applied to $P$, $P$ should satisfies $\widetilde{X}(P)=1$ by Fig.~\ref{ych25}, which implies contradiction.  Then 
\[P~{\text{can be related to $\bigcircle$ by using $RI$ and weak $RI\!I\!I$}} \Rightarrow \widetilde{X}(P)=\widetilde{X}(\bigcircle)=0. \]
Then 
\begin{equation*}\label{eq2}
\widetilde{X}(P)=0 \Leftrightarrow P~{\text{can be related to $\bigcircle$ by using $RI$ and weak $RI\!I\!I$}}
\end{equation*}
where we used the fourth formula to show ($\Rightarrow$).  

Finally, we show the second formula.  We assume that $\tr(P)=\lambda(P)=0$.  This condition implies $0$ $=$ $4\lambda(P)$ $=$ $3\h(P)$ $+$ $\otimes(P)$ and, originally, we have $\h(P) \ge 0$ and $\otimes(P) \ge 0$.  Then $\otimes(P)=\h(P)=0$.  We notice that $\otimes(P)=0$ if and only if $\widetilde{X}(P)=0$, which implies $\h(P)=0$.  Then, 
\[\tr(P)=\lambda(P)=0 \Leftrightarrow \otimes(P)=\h(P)=\lambda(P)=0 \Leftrightarrow \otimes(P)=0 \Leftrightarrow \widetilde{X}(P)=0. \]
Using the proof of the third formula in the above,  
\[\tr(P)=\lambda(P)=0 \Leftrightarrow P~{\text{can be related to $\bigcircle$ by using $RI$.}}\]
%Here, ($\Leftarrow$) is held since $\tr(P)$ and $\lambda(P)$ are invariant under RI.  
That completes the proof.  
%That completes the proof.  
%\begin{remark}
%Looking at (\ref{eq1}) and (\ref{eq2}), we notice the following.  
\end{proof}

The third and fourth formulae in Theorem~\ref{main3} imply \cite[Corollary 4.1]{IT}.  
\begin{corollary}[\cite{IT}]
A knot projection $P$ can be related to $\bigcircle$ by a finite sequence consisting of RI and weak RI\!I\!I if and only if $P$ can be related to $\bigcircle$ by a finite sequence consisting of RI.   
%\end{remark}
\end{corollary}
\begin{remark}
The above proof of the first formula in Theorem~\ref{main3} implies Fact~\ref{itt_thm} from \cite{ITT}.
\begin{fact}[Ito-Takimura-Taniyama \cite{ITT}]\label{itt_thm}
The following $(\ref{1})$ and $(\ref{2})$ are mutually equivalent.
\begin{enumerate}
\newcommand{\enumi}{(\arabic{1})}
\item A knot projection $P$ is any connected sum of knot projections, each of which is an element of $\mathcal{T}_1$. \label{1}
\item A knot projection $P$ and a simple closed curve $\bigcircle$ on the sphere can be related by a finite sequence consisting of RI and strong RI\!I\!I.  \label{2}
\end{enumerate}
\end{fact} 
\end{remark}
%\begin{remark}\label{re_arno}
%Arnaud Mortier told one of the author Ito one more other proof of \cite[Corollary 4.1]{IT} ($\ast$) as follows.   
%\begin{proof}
%Consider a finite sequence consisting of RI and weak RI\!I\!I from th simple closed curve $\bigcircle$ to a knot projection $P$.  Assume that the sequence contains weak RI\!I\!I.  Before the first appearance of weak RI\!I\!I, there is no $\otimes$.  Then, we cannot apply weak RI\!I\!I as shown in Fig.~\ref{ych25} to the knot projection.  Then the assumption that the sequence contains weak RI\!I\!I is false.  
%Using induction with respect to the number $n$ of local moves, RI and weak RI\!I\!I.  For $n=0$, the claim ($\ast$) holds.  Under assumption 
%\end{proof}
%\end{remark}
As in the proof of Theorem~\ref{main3}, we have
\begin{proposition}\label{additive_prop}
Let $P$ and $P'$ be arbitrary knot projections and $P \sharp P'$ the connected sum of $P$ and $P'$.  Let $x(P)$ be the number of sub-chord diagrams of type $x$ embedded into $CD_P$ for a knot projection $P$, where $x$ does not contain a chord that does not intersect any other chords. 
\[x(P \sharp P')=x(P)+x(P').\]
As a corollary, 
\[\lambda(P \sharp P') = \lambda(P) + \lambda(P').\]
\end{proposition}
\begin{proof}
By the definitions of a chord diagram of a knot projection and $x$, we immediately have $x(P \sharp P')=x(P)+x(P')$, since chords from $CD_P$ and those of $CD_{P'}$ do not intersect.  
%since chords of $P$ and chords of $P'$ do not intersect.  
%The chord diagram $CD_{P \sharp P'}$ also contains chords of $$
In particular, we consider the cases $x=\otimes, \tr, \h$.  Then, 
\begin{align*}
\lambda(P \sharp P') &= \frac{3}{4}\h(P \sharp P') - \frac{3}{4}\tr(P \sharp P') + \frac{1}{4} \otimes(P \sharp P') \\
&=\frac{3}{4}\h(P) + \frac{3}{4}\h(P') - \frac{3}{4}\tr(P) - \frac{3}{4}\tr(P') + \frac{1}{4} \otimes(P) + \frac{1}{4} \otimes(P')\\
&=\lambda(P) + \lambda(P').  
\end{align*}
\end{proof}
\begin{proposition}
For any integer $k$, there exists a knot projection $P$ such that $\lambda(P)=k$.  
\end{proposition}
\begin{proof}
In this proof, we use the symbol as $n_m$ to represent a knot projection defined by Figs.~\ref{hyou3} and \ref{hyou4}.  Note that
\[\lambda(4_1)=4, \lambda(5_1)=-5,~{\text{and}}~\lambda(7_3)=-3.\]
Using the above additivity, $\lambda(4_1 \sharp 5_1)=-1$ and $\lambda(4_1 \sharp 7_3)=1$.  If $k$ is a positive integer, a connected sum $P$ of $k$ pairs, each of which consists of $4_1$ and $7_3$ satisfies $\lambda(P)=k$.  If $k$ is a negative integer, a connected sum $P'$ of $k$ pairs, each of which consists of $4_1$ and $5_1$ satisfies $\lambda(P')=k$.  Noting that $\lambda(3_1)=0$, the proof is completed.  
\end{proof}
\section{Proof of Theorem~\ref{flype_thm}.}\label{sec_flype}
%First, we proves the former part.  
\begin{proof}
We consider all possibilities of connections of tangles shown in Fig.~\ref{flyped} having four end points.  
%(we often call the figure show in Fig.~\ref{flyped} {\it{the (2, 2)-tangles}}).  
Then, we have exactly six cases (Fig.~\ref{case1--6}).  
It is easy to see that we can omit Cases 2, 4, and 6.  Moreover, by retaking the shaded part, Case 3 becomes Case 5, and thus, we can omit Case 3.  Therefore, in the following, we only consider Cases 1 and 5.  
%We first show the claim in Case 1 and provide comments for the other cases, because the proofs of Cases 2-- Case 6 are very similar to the proof of Case 1.  
Throughout this proof, the phrase {\it{shaded parts}} refers to the shaded parts in Cases 1 and 5 in Fig.~\ref{case1--6}.  
\begin{figure}
\begin{tabular}{|c|c|}\hline
\includegraphics[width=6cm]{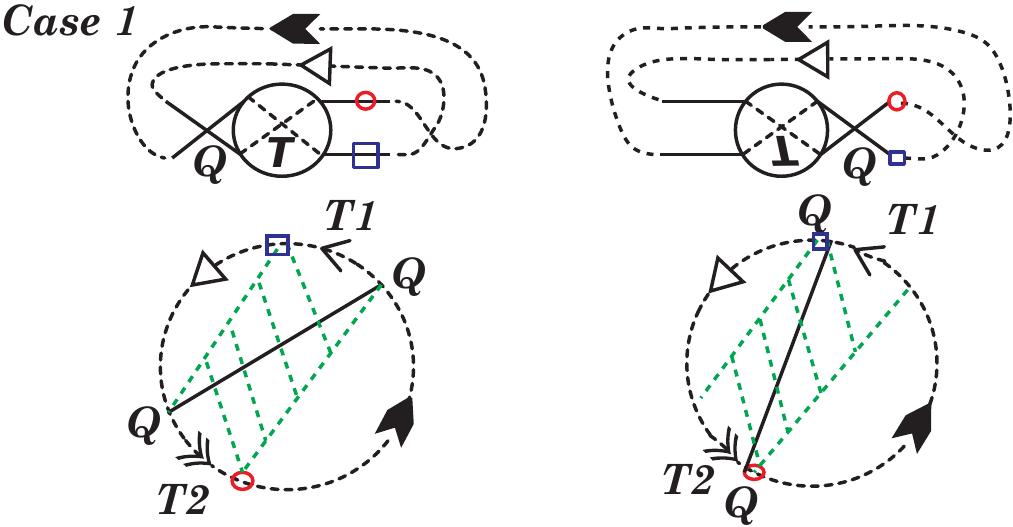}
& \includegraphics[width=6cm]{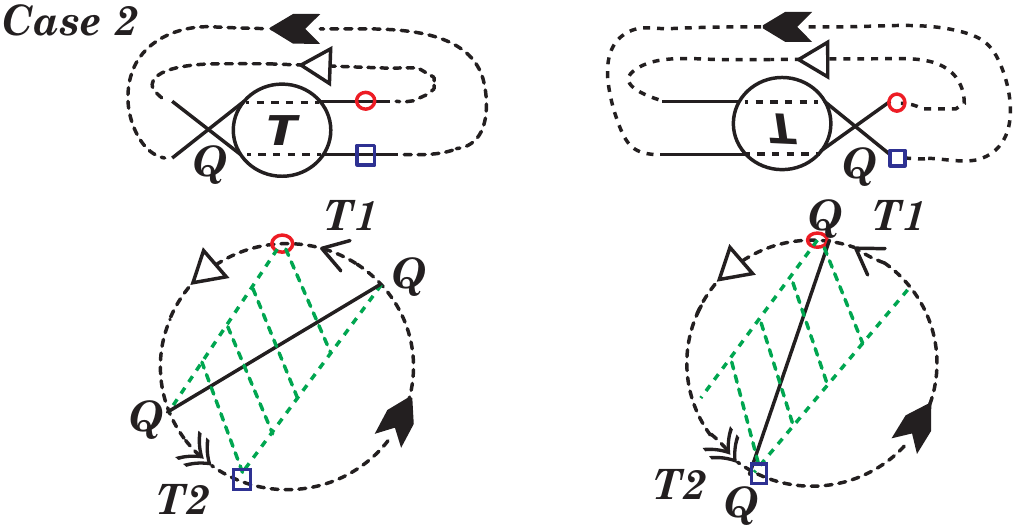} \\ \hline
\includegraphics[width=6cm]{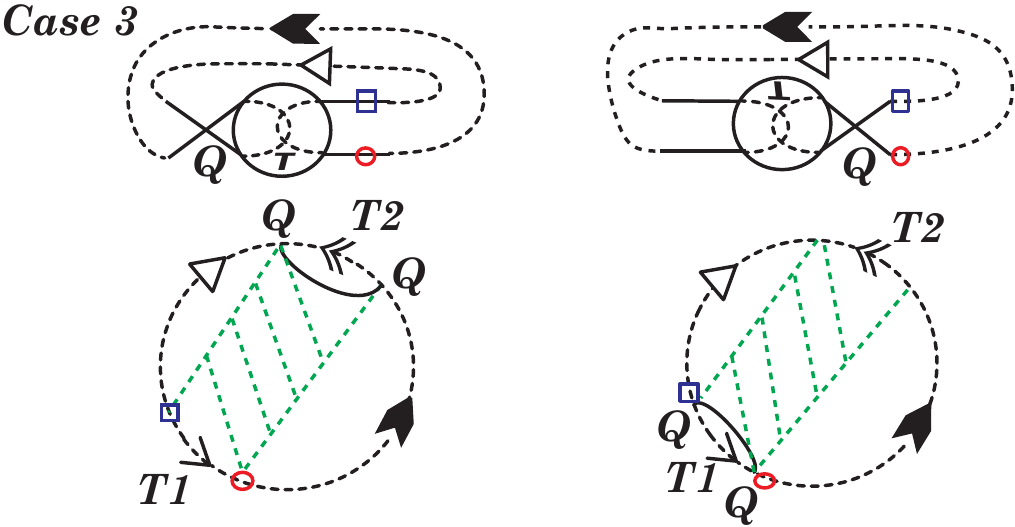}
& \includegraphics[width=6cm]{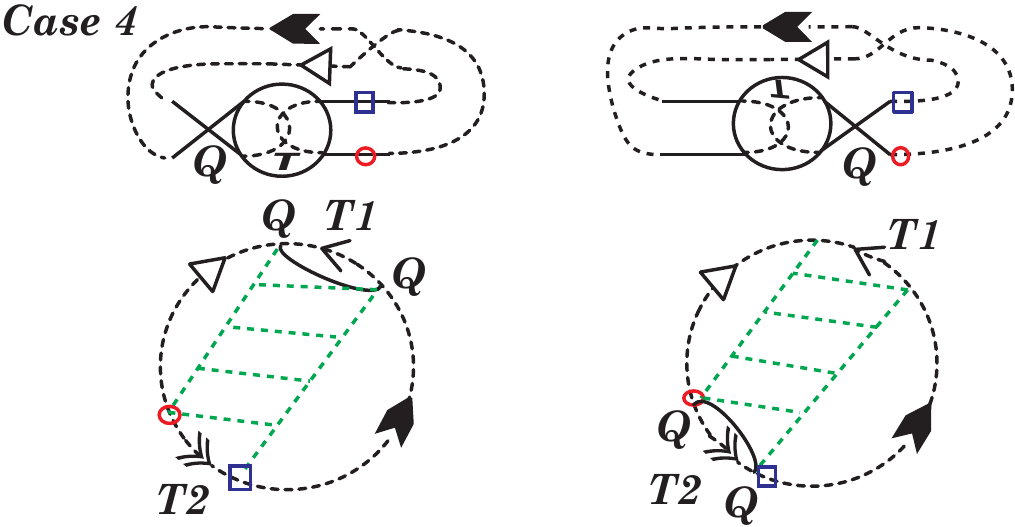} \\ \hline
\includegraphics[width=6cm]{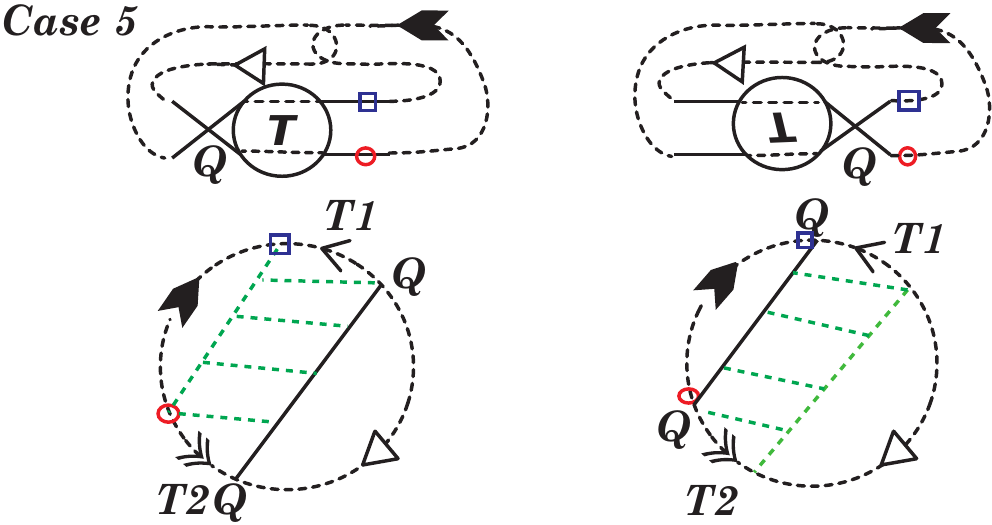}
& \includegraphics[width=6cm]{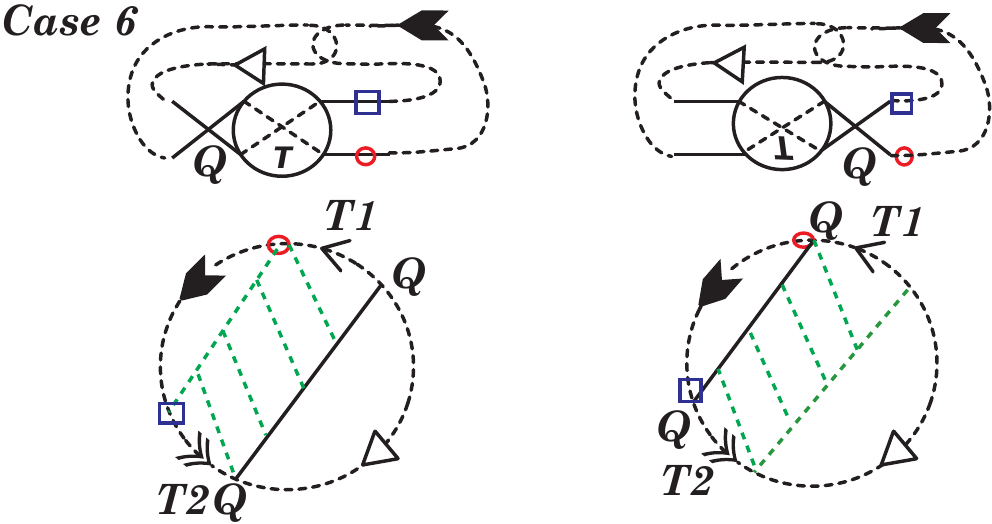} \\ \hline
\end{tabular}
\caption{Cases 1--6.}\label{case1--6}
\end{figure}
%\begin{itemize}
%\item Case 1.  
%In this case, one flype causes movement between the shaded part on the left and that on the right (Case 1 of Fig.~\ref{case1--6}).   The right shaded part is the mirror image of the left one.  
According to the definition of flypes, there is no chord connecting an endpoint on the shaded part with another endpoint on the non-shaded part ($\star$).  This condition is called {\it{condition $(\star)$}}.  Please refer to the following on the basis of Cases 1 and 5 in Fig.~\ref{case1--6}.  Further, note that we can omit the case wherein $Q$ is not contained by the sub-chord that is being counted, since such a sub-chord should be counted in both before and after applying a flype ($\diamondsuit$).  
\begin{itemize}
\item $\otimes$.  
\begin{itemize}
\item Assume that no chord of $\otimes$ is in the shaded part (zero-chord case).  Chord $Q$ is in the shaded part, and therefore, we can omit the case (cf.~($\diamondsuit$)).  
\item Assume that exactly one chord of $\otimes$ is in the shaded part (one-chord case).  Fig.~\ref{case1--6} (Cases 1 and 5) illustrates the claim.  
\item Assume that two chords of $\otimes$ are in the shaded part (two-chord case).    It is easy to verify the claim in Case 1.  Further, note that Case 5 is a case that is not related to $Q$ (cf.~($\diamondsuit$)).  
\end{itemize} 
%In particular, the discussions of the zero-chord case is very easy, since there are no changes counting the chords.  
\item $\tr$.  We present a discussion similar to that of $\otimes$ case.  In the rest of this proof, in the case that $y$ chords of the sub-chord we have chosen (now we choose $\tr$) are included in the shaded part in the whole chord diagram, we call the case a ``$y$-chord case.''
\begin{itemize}
\item Zero-chord case.  The $\tr$ we focus on is not contained in the shaded part and the case is not related to $Q$.  Therefore, we can omit the case.  
\item One-chord case.  Fig.~\ref{case1--6} (Cases 1 and 5) illustrates the claim.  
\item Two-chord case.  Case 5 has no possibility to realize $\tr$ containing $Q$.  Fig.~\ref{case1--6} (Case 1) illustrates the claim.
%\item Three chord case.  Case 5 has no possibility to realize $\tr$ containing $Q$.  In Case 1, it is easy to verify the claim.  
\item Three-chord case.  Case 5 has no possibility to realize $\tr$ containing $Q$.  Case 1 fixes the type of the two other chords crossing $Q$ and we therefore show the claim easily.   
%$\tr$ we focus on are contained by the non-shaded part and then does not change under the flype.  
\end{itemize}
We know that we need not mention the zero-chord case since the case has no possibility to realize the focused sub-chord containing $Q$.  
%it is easy to say the invariance of the number of sub-chord diagrams in the case that the shaded-part contains $0$ chord or $1$ chord.  
Therefore, we omit the zero-chord case in the following.  
%By the same reason, we can omit the case that non-shaded part contains $0$ chord or $1$ chord.
\item $\h$.  
$\h$ is composed of two parallel chords and the chord crossing the two parallel chords, called the {\it{sticking chord}}.   
\begin{itemize}
\item One-chord case.  Fig.~\ref{case1--6} (Cases 1 and 5) illustrates the claim.  In each case, chord $Q$ can either be the sticking chord or a non-sticking chord.  
%\item $1$ chord case.  If any one chord is contained the shaded part, the mirror image should be count it as $\h$.  Then we have the claim in the case.  
\item Two-chord case.  Condition ($\star$) and the case begin considered require that there be no possibility realizing $\h$ containing $Q$ in Case 1; then, we do not need to consider the case.  In Case 5, $Q$ cannot be the sticking chord; it is easy to verify the claim.  
\item Three-chord case.  Case 5 has no possibility to realize $\h$ containing $Q$.  Case 1 fixes $\h$ containing $Q$, which becomes the sticking chord, and therefore, we easily obtain the claim.  
% the shaded part should contain the two parallel chords.  Then we have the claim in the case.  
\end{itemize}
%We notice that $1$ chord case can be omitted if we discuss any cases other than the case $\h$.  
\item $\thr$.  
%We can omit $0$ chord, $1$ chord, $3$ chord and $4$ chord cases.  
\begin{itemize}
\item One-chord case.  Through Fig.~\ref{case1--6} (Cases 1 and 5), it is easy to verify the claim.  
\item Two-chords case.  We notice that the shaded part must contain two adjacent chords of $\thr$ by the condition ($\star$).  Accordingly, there is no possibility to realize $\thr$ containing $Q$ in Case 1.  
Simlarly, we have the claim in this case using Fig.~\ref{case1--6} (Case 5).    
%\item $3$ chord case.  We notice that the shaded part should contain three adjacent chords of $\thr$.  Then we have the claim in this case.  
\item Three-chord case.  Case 1 has no possibility to realize $\thr$ containing $Q$.  Case 5 has three parallel chords containing $Q$ in the shaded part and it is easy to verify the claim.  
\item Four-chord case.  Case 5 has no possibility to realize $\thr$ containing $Q$.  In Case 1, $Q$ must cross three parallel chords that fix $\thr$, and thus, we have the claim.   
\end{itemize}
\item $\HH$.  
%We can omit $0$ chord, $1$ chord, $3$ chord, and $4$ chord cases.  
\begin{itemize}
\item One-chord case.  Fig.~\ref{case1--6} (Cases 1 and 5) illustrates the claim.  
\item Two-chord case.  In this case, two parallel strands of $\HH$ must be in the shaded part.  Then, Case 1 has no possibility to realize $\HH$ containing $Q$.  In Case 5, the claim is easily shown.  
\item Three-chord case and Four-chord case.  Condition ($\star$) requires that there be no possibility to realize $\HH$ containing $Q$ in Cases 1 and 5.  
%There is no other possibility in this case. 
%\item $3$ chord case.  There is only one chord of $\HH$ belongs to non-shaded part.  
\end{itemize}
\end{itemize}
%\item Case 2. One flype causes that the shaded part moves to its mirror image (Fig.~\ref{case1--6} Case 2).   
%\item Case 3. One flype causes that the shaded part is rotated by $180$ degree (Fig.~\ref{case1--6} Case 3).  
%\item Case 4. One flype causes that the shaded part moves to its mirror image (Fig.~\ref{case1--6} Case 4).  
%\item Case 5. One flype causes that the shaded part moves to its mirror image (Fig.~\ref{case1--6} Case 5).  
%\item Case 6.  One flype causes that the shaded part is rotate by $180$ degree (Fig.~\ref{case1--6} Case 6).  
%\end{itemize}
\end{proof}
\begin{remark}
In Fig.~\ref{hyou5}, there are three pairs $(7_A, 7_6)$, $(7_B, 7_7)$, and $(7_C, 7_5)$ with respect to flypes.  In each pair, one can be related to another by one flype.  
\end{remark}

%\section{Another proof of Theorem on the first and weak RI\!I\!I.}

\section{Relationship of the number of sub-chord diagrams and Arnold invariants.}\label{average_sec}
This section contains comments regarding the relationship between our study and Arnold invariants.  
Theorem \ref{thm_sec4} counts the number of sub-chord diagrams in $CD_P$ of a knot projection $P$.  In comparison, for the Arnold invariants $J^{+}$, $J^-$, and $St$, Averaged invariant $-(J^{+} + 2St)/2$ counts the sum of signs $\pm 1$, where each sign is assigned to a sub-chord $\otimes$ (further details can be found in \cite{polyak1998}; note also that Polyak's original Averaged invariant is $(J^{+} + 2 St)/8$).  Let $P$ be an arbitrary knot projection (the image of an immersion) on $S^2$.  Putting $\infty$ on this arbitrarily selected region $r(\infty)$ from $S^{2} \setminus P$, $P$ can be regard as a plane immersed curve and is denoted by $P_r(\infty)$.  Arnold invariants, $J^+$, $J^-$, and $St$, are defined for plane immersed curves.  Proceeding further, $J^{+}(P_r(\infty) + 2 St(P_{r(\infty)}))$ does not depend on the selection of $r(\infty)$ (see \cite[Sec.~2.4]{polyak1998}).  Thus, we have an integer $J^{+}(P) + 2 St(P)$ for an arbitrary spherical curve $P$.  Then, Averaged invariant $a(P)$ is defined by
\[a(P) = -(J^{+}(P) + 2 St(P))/2.  \]
Recall that any two knot projections $P_1$ and $P_2$ are related by a finite sequence of three types of Reidemeister moves, as shown in Fig.~\ref{ych5}.  The definition of $a(P)$ implies (\ref{p1})--(\ref{p3}).  
%The definitions of $J^+$ and $St$ directly deduce (\ref{p5}), (\ref{p2}) and (\ref{p3}).  
\begin{remark}
%{\UTF{0083}v\UTF{0083}\UTF{008A}\UTF{0083}\UTF{0093}\UTF{0083}g\UTF{0082}\UTF{00C9}\UTF{0082}\UTF{00A0}\UTF{0082}\UTF{00E9}\UTF{0089}\UTF{00F1}\UTF{0090}\UTF{0094}\UTF{0082}\UTF{00C9}\UTF{0082}\UTF{00A0}\UTF{0082}\UTF{00E9}\UTF{008C}\UTF{008B}\UTF{0089}\UTF{00CA}}
Let $m$ be the total number of weak second and third Reidemeister moves in a finite sequence consisting of first, second, and third Reidemeister moves between two knot projections $P_1$ and $P_2$.  
\begin{enumerate}
\item $a(P_1) - a(P_2)$ $\equiv$ $m$ (mod $2$) \label{p1}, 
\item $a(P)$ is invariant under RI, \label{p4}
\item $a(P)$ is invariant under strong RI\!I \label{p5}
\item a single RI\!I\!I changes $a(P)$ by $\pm 1$,  \label{p2}
\item a single weak RI\!I changes $a(P)$ by $\pm 1$.  \label{p3}
\end{enumerate}
\end{remark}
\begin{proof}
The definitions of $J^+$ and $St$ immediately imply (\ref{p5}), (\ref{p2}), and (\ref{p3}).  Thus, if we have (\ref{p4}), then we have (\ref{p1}).  Here, we recall Polyak's formula for $a(P)$, which directly implies (\ref{p4}).  

Let $X^{*}$ be a $\otimes$ with a base point on $S^1$ of $\otimes$ apart from any endpoints of the two chords.  Similarly, a chord diagram with a base point is denoted by $CD^{*}_P$ which is defined as a chord diagram $CD_P$ with a point on $S^{1}$ except for any endpoints of chords.  Note that the orientation of $P$ having the base point, which is on the curve except for double points, corresponds to the orientation of $CD^{*}_P$ when we always orient $S^{1}$ of $CD_P$ counterclockwise.  Along the lines of \cite[Sec.~6.4]{polyak1998}, we recall Polyak's formulation of $a(P)$ as follows.  

Let us obtain any orientation of $S^{2}$ and any orientation of a knot projection $P$.  We start from the base point and move along the orientation $P$.  Each time we pass through a double point for the first time, we attach a sign ($= -1, 1$).  For each double point $\double$ through which branch $t$ passes, we assign a pair $(t_1, t_2)$ that indicates the orientation rotating from $t_1$ to $t_2$.   If the orientation $(t_1, t_2)$ is (resp.,~is not) equal to the orientation $S^2$, the sign of the double point is $-1$ (resp.,~$1$).  For instance, choosing appropriate orientations of the sphere, we describe the sign simply as follows.  For each double point $\double$ having branches $t_1$ and $t_2$, where $t_1$ (resp.,~$t_2$) is the branch we pass through when we pass through the double point for the first (resp.,~second) time, if $t_1$ is the arrow from the bottom left to the top right, the double point has sign $=$ $-1$ and if not, the sign is $1$.  Assign each sign of a double point to each corresponding chord.  Then sub-chord $X^{*}$ embedded into $CD^{*}_P$ has two signs $\epsilon_0(X^{*})$ and $\epsilon_1(X^{*})$ for each $X^{*}$.  As in \cite[Page 997, Formula (3)]{polyak1998}, we can show
\[\sum_{ X^{*}~{\text{embedded in $CD^{*}_P$}} } \epsilon_0(X^{*}) \epsilon_1(X^{*}) = a(P).\]
%Now we show the statement of (\ref{p1}).  
By the above formula, the first Reidemeister move does not affect $a(P)$.  That completes the proof.
%Second, we consider the single strong second Reidemeister move producing two chords $c_0$ and $c_1$.  If $c_0$ and $c'_0 (\neq c_1)$ matches $\otimes$ and contributes to $a(P)$ as $\eta= -1$ pr $1$, $c_1$ and $c'_0$ contributes $a(P)$ to $-\eta$ and the total affection is $\eta - \eta=0$.  Then, we have (\ref{p4}).  Third, we consider the single weak second Reidemeister move producing two chords $c_0$ and $c_1$.  Similarly to the strong second Reidemeister move, if $c'_0 (\neq c_1)$ and $c_0$ matches $\otimes$, the contribution is canceled out by a pair $c'_0$ and $c_1$.  Thus, only $\epsilon_0 \epsilon_1 = -1$ contribute to $a(P)$ under the weak second Reidemeister move increasing double points.  This implies the claim (\ref{p2}).  Fourth, we consider the third Reidemeister move.  Using the formula $a(P) = -J^{+}(P)/2 - St(P)$, $a(P)$ is changed by $\pm 1$ under the single third Reidemeister move.  This completes the proof of (\ref{p3}).  In summery, we have the proof of the claim (\ref{p1}).  
\end{proof}
\begin{remark}
We present a table of the values of Averaged invariants $a(P)$ $=$ $-(J^+(P) + 2St(P))/2$ for prime reduced knot projections with up to seven double points, which appear in Figs.~\ref{hyou3} and \ref{hyou4}.

\noindent
\scalebox{0.88}{
\begin{tabular}{|c|c|c|c|c|c|c|c|c|c|c|c|c|c|c|c|c|c|c|} \hline\label{average_table} 
\!\!$\bigcircle$&$3_1$ & $4_1$ & $5_1$ & $5_2$ & $6_1$ & $6_2$ & $6_3$ & $7_1$ & $7_2$ & $7_3$ & $7_4$ & $7_5$ & $7_6$ & $7_7$ & $7_A$ & $7_B$ & $7_C$ \\  
$0$&$-1$ & $0$ & $-2$ & $-1$ & $0$ & $-1$ & $-2$&$-3$&$-1$&$-2$&$-1$&$-2$&$-1$& $0$ & $-1$ & $0$ & $-2$\\ \hline
\end{tabular}
}
\cite{IT2} introduced another integer-valued additive invariant and a complete invariant for prime reduced knot projections with up to seven double points except for one pair under an equivalence relation determined by RI and strong RI\!I.  
\end{remark}

\section{Tables of knot projections with invariants.}
Finally, we present two tables.  The first consists of Figs.~\ref{hyou3} and \ref{hyou4}.   The table contains prime reduced knot projections with up to seven double points and their chord diagrams, each of which has the number of cross chords, triple chords, H-chords, $\HH$-type sub-chord diagrams, or I\!I\!I-chords embedded in the chord diagram of the knot projection.  Fig.~\ref{hyou5} shows prime reduced knot projections with integers that are the values of $\lambda$.  Note that prime knot projections lacking only the knot projection $\infty$ are prime reduced knot projections.  Here, a prime knot projection is defined as a knot projection that cannot be represented as the connected sum of two non-trivial knot projections.  The second table consists of prime reduced knot projections with up to seven double points with $\lambda$.  For two knot projections $P_1$ and $P_2$, we connect $P_1$ and $P_2$ with a line in the table, if $P_1$ can be related to $P_2$ using RI and strong RI\!I\!I under the following rule: except for a pair $(7_B, 7_4)$, every line indicates the existence of a sequence of a finite number of RIs and a strong RI\!I\!I (Fig.~\ref{hyou5}).  
\begin{figure}[h!]
\includegraphics[width=12cm]{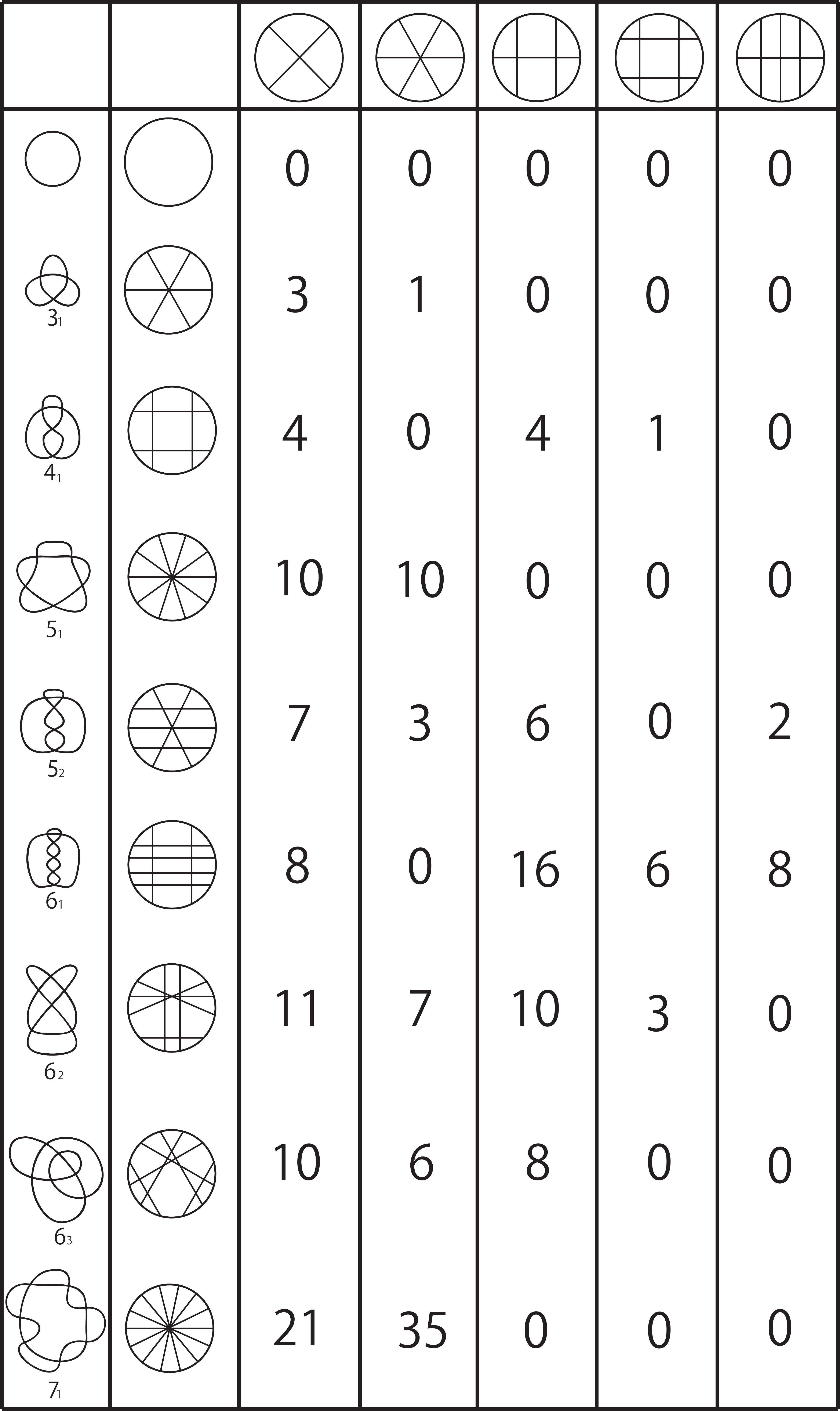}
\caption{Table 1: knot projections from a simple closed curve to $7_1$.}\label{hyou3}
\end{figure}
\begin{figure}[h!]
\includegraphics[width=12cm]{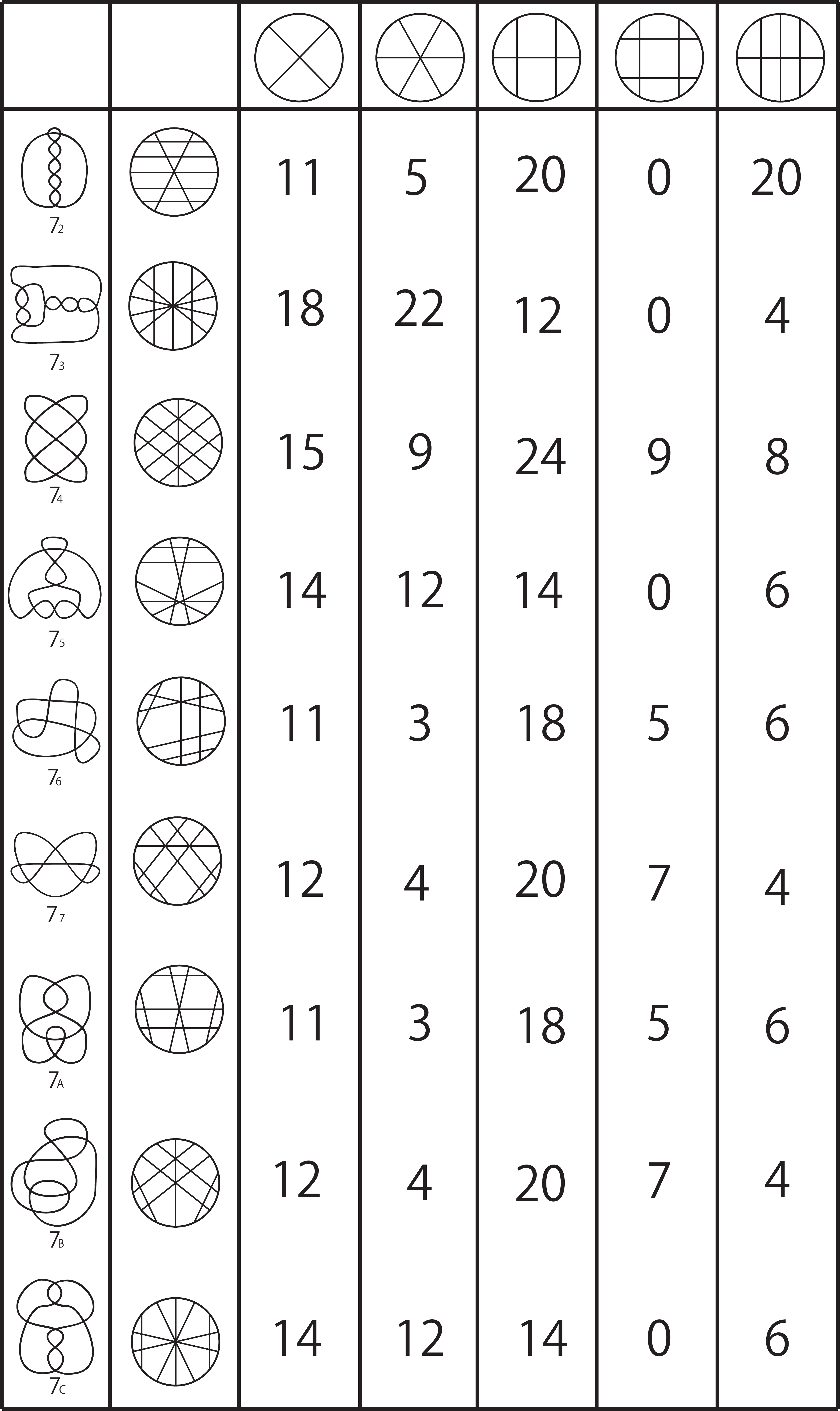}
\caption{Table 2: knot projections $7_2$--$7_C$.  
%Table of the numbers of cross chords, triple chords, H-cords, and three chords as sub-chords embedded in chord diagrams of knot projections with small double points.
}\label{hyou4}
\end{figure}
\begin{figure}[h!]
\includegraphics[width=12cm]{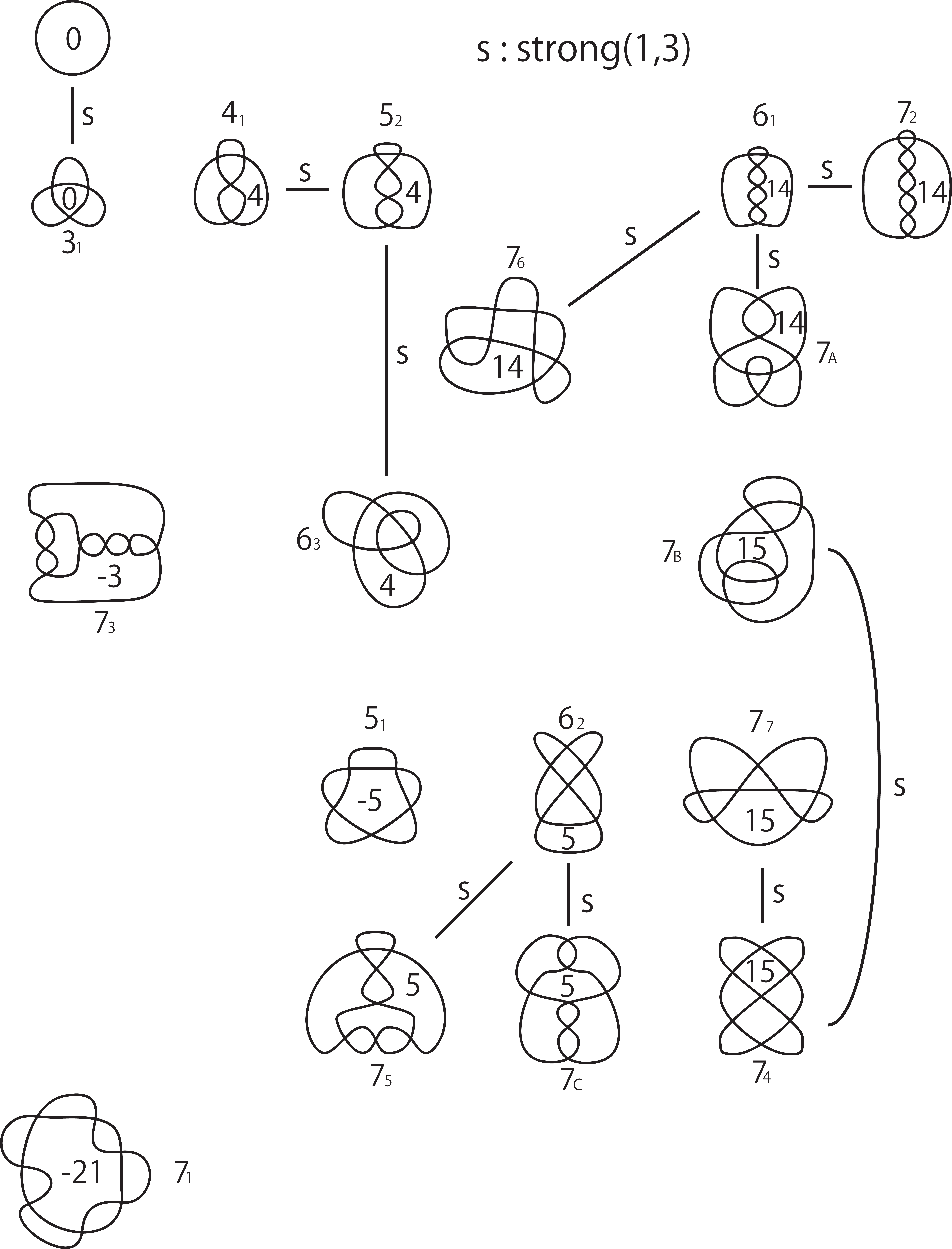}
\caption{Table of prime reduced knot projections with up to seven double points with $\lambda$.  Integers with knot projections stand for $\lambda$.  
%$7_7$ is related to $7_B$ by one flype.
}\label{hyou5}
\end{figure}

\section*{Acknowledgements}
The authors would like to thank Professor Kouki Taniyama for his helpful comments.  The authors would also like to thank the referee for their comments on earlier versions of this paper.  The work of N. Ito was partly supported by a Waseda University Grant for Special Research Projects (Project number: 2014K-6292) and the JSPS Japanese-German Graduate Externship.  
%N. Ito would like to thank Arnaud Mortier for his Remark \ref{re_arno} to \cite[Corollary 4.1]{IT} in N-KOOK seminar in Osaka.  


\begin{thebibliography}{99}
\bibitem{IT} Ito, N., Takimura, Y.: (1, 2) and weak (1, 3) homotopies on knot projections, J. Knot Theory Ramifications {\bf{22}} (2013), 1350085 (14 pages).  
\bibitem{IT2} Ito, N., Takimura, Y.: Strong and weak (1, 2) homotopies on knot projections and new invariants, to appear in Kobe J. Math.
\bibitem{ITT} Ito, N., Takimura, Y., Taniyama, K.: Strong and weak (1, 3) homotopies on knot projections, to appear in Osaka J. Math.
\bibitem{polyak1998} Polyak, M.: Invariants of curves and fronts via Gauss diagrams, \emph{Topology} {\bf{37}}, 989--1009.  
\bibitem{ST} Sakamoto, M., Taniyama, K.: Plane curves in an immersed graph in $\Bbb R\sp 2 $, J. Knot Theory Ramifications {\bf{22}} (2013), 1350003, 10pp.
\end{thebibliography}
\end{document}